\documentclass{amsart}
\usepackage{amsmath}
\usepackage{amssymb}
\usepackage{amsthm}
\usepackage{enumerate}
\usepackage{caption}
\theoremstyle{definition}
\newtheorem{definition}{Definition}[section]
\theoremstyle{plain}
\newtheorem{lemma}[definition]{Lemma}
\newtheorem{theorem}[definition]{Theorem}

\newtheorem{proposition}[definition]{Proposition}
\newtheorem{corollary}[definition]{Corollary}
\theoremstyle{remark}
\newtheorem{remark}[definition]{Remark}

\newtheorem{example}[definition]{Example}

\makeatletter
\@namedef{subjclassname@2020}{
  \textup{2020} Mathematics Subject Classification}
\makeatother

\newcommand{\myint}{\operatorname{int}}
\newcommand{\mycl}{\operatorname{cl}}
\newcommand{\mysucc}{\operatorname{succ}}
\newcommand{\mynlb}{\operatorname{NLB}}
\newcommand{\myfib}{\operatorname{fib}}

\begin{document}
\title[Notes on definably complete local o-minimality]{Notes on definably complete locally o-minimal expansions of ordered groups}
\author[M. Fujita]{Masato Fujita}
\address{Department of Liberal Arts,
Japan Coast Guard Academy,
5-1 Wakaba-cho, Kure, Hiroshima 737-8512, Japan}
\email{fujita.masato.p34@kyoto-u.jp}

\begin{abstract}
We study definably complete locally o-minimal expansions of ordered groups in this paper.
A definable continuous function defined on a closed, bounded and definable set behave like a continuous function on a compact set.
We demonstrate uniform continuity of a definable continuous function on a closed, bounded and definable set and Arzela-Ascoli-type theorem.
We propose a notion of special submanifolds with tubular neighborhoods and show that any definable set is decomposed into finitely many special submanifolds with tubular neighborhoods.
\end{abstract}



\maketitle

\section{Introduction}\label{sec:intro}
We study definably complete locally o-minimal expansions of ordered groups in this paper.
An o-minimal structure enjoys tame properties such as monotonicity and definable cell decomposition \cite{vdD,KPS,PS}. 
Toffalori and Vozoris studied locally o-minimal structures in \cite{TV}.
Roughly speaking, a locally o-minimal structure is defined by simply localizing the definition of an o-minimal structure.
However, their study reveals that the local version of monotonicity theorem is unavailable in a locally o-minimal structure.
Local o-minimal structures are studied also in \cite{KTTT}.

The above two papers do not assume definable completeness.
Fornaisero made a comprehensive study on definably complete locally o-minimal expansions of ordered fields in \cite{F}.
This study showed that definably complete locally o-minimal structures enjoy tame topological properties when they are expansions of ordered fields.
The author and his collaborators have demonstrated that definably complete locally o-minimal structures still enjoy tame topological properties without algebraic assumptions such as that they are expansions of ordered fields \cite{Fuji, Fuji2, Fuji4, FKK}. 
These are introduced in Section \ref{sec:preliminary}.
As a preliminary, we also demonstrate a definable choice lemma for a definably complete expansion of an ordered group and its immediate consequences in this section.

In this paper, as we said at the beginning of the paper, we employ a weak algebraic assumption; that is, we consider definably complete locally o-minimal expansions of ordered groups.
We first demonstrate that a definable continuous function defined on a closed, bounded and definable set behave like a continuous function on a compact set in Section \ref{sec:function}.
More precisely, definable version of uniform continuity of continuous functions on closed, bounded and definable sets and definable Arzela-Ascoli-type theorem are demonstrated.

The notion of quasi-special submanifolds is introduced in \cite{Fuji4}.
Any set definable in a definably complete locally o-minimal structure is decompoed into finitely many quasi-quadratic submanifolds \cite[Proposition 2.11]{FKK}.
This notion is obtained by relaxing the definitions of multi-cells in \cite{F} and special submanifolds in \cite{M2}, which define the same notion in our setting as we demonstrate later.
Miller, Thamrongthanyalak and Fornasiero demonstrated that a definable set is decomposed into finitely many special submanifolds/multi-cells under the assumption that the structure is an expansion of an ordered field \cite{F,M2,T}.
(Miller and Thamrongthanyalak assumed d-minimality instead of local o-minimality.)
In Section \ref{sec:decomposition}, we extend the definition of special submanifolds to the case in which the structures are expansions of dense linear orders without endpoints.
We prove that any set definable in a definably complete locally o-minimal expansion of an ordered group is decomposed into finitely many special submanifolds.
In addition, we introduce the notion of special submanifolds with tubular neighborhoods and demonstrate a decomposition theorem into them.

We introduce the terms and notations used in this paper.
The term `definable' means `definable in the given structure with parameters' in this paper.
For a linearly ordered structure $\mathcal M=(M,<,\ldots)$, an open interval is a definable set of the form $\{x \in R\;|\; a < x < b\}$ for some $a,b \in M \cup \{\pm \infty\}$.
It is denoted by $(a,b)$ in this paper.
Elements in $M^2$ are denoted by the same notation, but it will not confuse readers.
We define a closed interval similarly. 
It is denoted by $[a,b]$.
An open box in $M^n$ is the direct product of $n$ open intervals.
We set $$\mathcal B_m(x,\varepsilon)=\{y=(y_1,\ldots, y_m) \in M^m\;|\; |x_i-y_i|<\varepsilon \text{ for all }1 \leq i \leq m\}$$ for any $x=(x_1,\ldots, x_m) \in M^m$ and $\varepsilon>0$.
The notation $M_{>r}$ denotes the set $\{x \in M\;|\;x>r\}$ for any $r \in M$.
We set $|x|:=\max_{1 \leq i \leq n}|x_i|$ for any vector $x = (x_1, \ldots, x_n) \in M^n$ when the addition is definable in $\mathcal M$.
The function $|x-y|$ defines a distance in $M^n$ when $\mathcal M$ is an expansion of an ordered abelian group.
Let $\myint(A)$ and $\mycl(A)$ denote the interior and the closure of a subset $A$ of a topological space, respectively.

\section{Preliminary}\label{sec:preliminary}
\subsection{Results in previous studies}

We first recall basic definitions.
\begin{definition}\label{def:dctc}
An expansion of a  dense linear order without endpoints $\mathcal M=(M,<,\ldots)$ is \textit{definably complete} if every definable subset of $M$ has both a supremum and an infimum in $M \cup \{ \pm \infty\}$ \cite{M}.

An expansion of a  dense linear order without endpoints $\mathcal M=(M,<,\ldots)$ is \textit{locally o-minimal} if, for every definable subset $X$ of $M$ and for every point $a\in M$, there exists an open interval $I$ containing the point $a$ such that $X \cap I$ is  a finite union of points and open intervals \cite{TV}.

A definably complete locally o-minimal structure is \textit{a model of DCTC} if any definable discrete subset of $M$ is bounded.
\end{definition}
The definition given above is not the same as the original definition of a model of DCTC in \cite{S}.
But they are equivalent by \cite[Corollary 2.8]{S}.

We also recall the definition of local monotonicity.
\begin{definition}[Local monotonicity]
A function $f$ defined on an open interval $I$ is \textit{locally constant} if, for any $x \in I$, there exists an open interval $J$ such that $x \in J \subseteq I$ and the restriction $f|_J$ of $f$ to $J$ is constant.
A function $f$ defined on an open interval $I$ is \textit{locally strictly increasing} if, for any $x \in I$, there exists an open interval $J$ such that $x \in J \subseteq I$ and $f$ is strictly increasing on the interval $J$.
We define a \textit{locally strictly decreasing} function similarly. 
A \textit{locally strictly monotone} function is a locally strictly increasing function or a locally strictly decreasing function.
A \textit{locally monotone} function is locally strictly monotone or locally constant.
\end{definition}

The following monotonicity theorem holds true.

\begin{theorem}[Monotonicity theorem]\label{thm:mono}
Let $\mathcal M=(M,<,\ldots)$ be a definably complete locally o-minimal structure.
Let $I$ be an interval and $f:I \rightarrow M$ be a definable function.
There exists a mutually disjoint partition $I=X_d \cup X_c \cup X_+ \cup X_-$ of $I$ into definable sets satisfying the following conditions:
\begin{enumerate}
\item[(1)] the definable set $X_d$ is discrete and closed;
\item[(2)] the definable set $X_c$ is open and $f$ is locally constant on $X_c$;
\item[(3)] the definable set $X_+$ is open and $f$ is locally strictly increasing and continuous on $X_+$;
\item[(4)] the definable set $X_-$ is open and $f$ is locally strictly decreasing and continuous on $X_-$.
\end{enumerate}
\end{theorem}
\begin{proof}
\cite[Theorem 2.3]{FKK}.
\end{proof}

We also need the following lemma in \cite{Fuji}:
\begin{lemma}\label{lem:dc_mono}
Let $\mathcal M=(M,<,\ldots)$ be a definably complete local o-minimal structure.
 A locally monotone definable function defined on an open interval is monotone.
\end{lemma}
\begin{proof}
\cite[Proposition 3.1]{Fuji}
\end{proof}

The following proposition guarantees the existence of the limit.

\begin{proposition}\label{prop:limit}
Let $\mathcal M=(M,<,0,+,\ldots)$ be a definably complete locally o-minimal expansion of an ordered group.
Let $s>0$ and $f:(0,s) \rightarrow M^n$ be a bounded definable map.
\begin{enumerate}
\item[(1)] There exists a positive $0<u<s$ such that the restriction of $f$ to $(0,u)$ is continuous. In addition, it is monotone when $n=1$;
\item[(2)] There exists a unique point $x \in M^n$ satisfying the following condition:
$$
\forall \varepsilon >0, \exists \delta>0, \forall t, \ 0<t<\delta \Rightarrow |x-f(t)| < \varepsilon \text{.}
$$ 
The notation $\lim_{t \to +0}f(t)$ denotes the point $x$.
\end{enumerate}
\end{proposition}
\begin{proof}
We first reduce to the case in which $n=1$.
Assume that the proposition holds true for $n=1$.
Let $\pi_i$ be the projection onto the $i$-th coordinate for all $1 \leq i \leq n$.
Apply the proposition to the composition $\pi_i \circ f$.
There exists $0<u_i<s$ such that the restriction of $\pi_i \circ f$ to $(0,u_i)$ is continuous and monotone for each $1 \leq i \leq n$.
Set $u=\min_{1 \leq i \leq n}u_i$.
The restriction of $f$ to $(0,u)$ is continuous.
Set $x_i=\lim_{t \to +0}\pi_i \circ f(t)$ for all $1 \leq i \leq n$.
It is obvious that $x=(x_1, \ldots, x_n)$ is the unique point satisfying the condition in the assertion (2).
We have succeeded in reducing to the case in which $n=1$.

Set $I=(0,s)$.
Applying Theorem \ref{thm:mono} to $f$, we get a partition $I=X_d \cup X_c \cup X_+ \cup X_-$ into definable sets such that $X_d$, $X_c$, $X_+$ and $X_-$ satisfy the conditions in Theorem \ref{thm:mono}.
Take a sufficiently small open interval $J$ containing the point $0$.
The intersections of $J$ with $X_d$, $X_c$, $X_+$ and $X_-$ are finite unions of points and open intervals.
Shrinking the interval $I$ if necessary, we have $I=X_c$, $I=X_+$ or $I=X_-$.
We have demonstrated the assertion (1) by Lemma \ref{lem:dc_mono}.
The remaining task is to show the assertion (2).
We only consider the case in which $I=X_-$.
We can prove the corollary similarly in the other cases.

The function $f$ is strictly decreasing because $I=X_-$.
Set $x=\inf_{0<t<s}f(t)$, which exists because $f$ is bounded.
It is obvious the point $x$ satisfies the required condition because $f$ is strictly decreasing.
Let $x'$ be another point satisfying the condition.
We fix an arbitrary $\varepsilon >0$.
There exists $\delta>0$ with $|x-f(t)|<\varepsilon$ whenever $0<t<\delta$.
There exists $\delta'>0$ with $|x'-f(t)|<\varepsilon$ whenever $0<t<\delta'$.
Set $\delta''=\min \{\delta,\delta'\}$.
We have $|x-x'| \leq |x-f(t)|+|x'-f(t)| < 2\varepsilon$ whenever $0<t<\delta''$.
We get $x=x'$ because $\varepsilon$ is an arbitrary positive element.
\end{proof}

\begin{definition}[Dimension]\label{def:dim}
Consider an expansion of a densely linearly order without endpoints $\mathcal M=(M,<,\ldots)$.
Let $X$ be a nonempty definable subset of $M^n$.
The dimension of $X$ is the maximal nonnegative integer $d$ such that $\pi(X)$ has a nonempty interior for some coordinate projection $\pi:M^n \rightarrow M^d$.
Here, we consider that $M^0$ is a singleton equipped with the trivial topology and the projection $\pi:M^n \rightarrow M^0$ is the trivial map.
We set $\dim(X)=-\infty$ when $X$ is an empty set.
\end{definition}

We only review the assertions on dimension which are frequently used in this study.

\begin{proposition}\label{prop:dim}
Consider a definably complete locally o-minimal structure $\mathcal M=(M,<,\ldots)$.
The following assertions hold true.
\begin{enumerate}
\item[(1)] A definable set is of dimension zero if and only if it is discrete.
When it is of dimension zero, it is also closed.
\item[(2)] Let $X \subseteq Y$ be definable sets.
Then, the inequality $\dim(X) \leq \dim(Y)$ holds true.
\item[(3)] Let $\sigma$ be a permutation of the set $\{1,\ldots,n\}$.
The definable map $\overline{\sigma}:M^n \rightarrow M^n$ is defined by $\overline{\sigma}(x_1, \ldots, x_n) = (x_{\sigma(1)},\ldots, x_{\sigma(n)})$.
Then, we have $\dim(X)=\dim(\overline{\sigma}(X))$ for any definable subset $X$ of $M^n$.
\item[(4)] Let $X$ and $Y$ be definable sets.
We have $\dim(X \times Y) = \dim(X)+\dim(Y)$.
\item[(5)] Let $X$ and $Y$ be definable subsets of $M^n$.
We have 
\begin{align*}
\dim(X \cup Y)=\max\{\dim(X),\dim(Y)\}\text{.}
\end{align*}
\item[(6)] Let $f:X \rightarrow M^n$ be a definable map. 
We have $\dim(f(X)) \leq \dim X$.
\item[(7)] Let $f:X \rightarrow M^n$ be a definable map. 
Let $\mathcal D(f)$ denote the set of points at which the map $f$ is discontinuous. 
The inequality $\dim(\mathcal D(f)) < \dim X$ holds true.
\item[(8)] Let $X$ be a definable set.
Let $\partial X$ denote the frontier of $X$ defined by $\partial X = \mycl(X) \setminus X$.
We have $\dim (\partial X) < \dim X$.
\item[(9)] Let $\varphi:X \rightarrow Y$ be a definable surjective map whose fibers are equi-dimensional; that is, the dimensions of the fibers $\varphi^{-1}(y)$ are constant.
We have $\dim X = \dim Y + \dim \varphi^{-1}(y)$ for all $y \in Y$.  
\item[(10)] Let $X$ be a definable subset of $M^n$.
There exists a point $x \in X$ such that we have $\dim(X \cap B)=\dim(X)$ for any open box $B$ containing the point $x$.
\end{enumerate}
\end{proposition}
\begin{proof}
See \cite[Proposition 2.8]{FKK}.
\end{proof}

\begin{corollary}\label{cor:point}
Consider a definably complete locally o-minimal structure $\mathcal M=(M,<,\ldots)$.
Let $X$ and $Y$ be definable sets such that $X \subseteq Y$ and $\dim X=\dim Y$.
Then, there exists a nonempty open box $B$ such that $X \cap B=Y\cap B$ and $\dim X \cap B=\dim X$.
\end{corollary}
\begin{proof}
	Set $d=\dim X$, $Z=\mycl(Y \setminus X)$, $V=X \cap Z$ and $W=X \setminus Z$.
	The set $V$ is contained in the frontier of $Y \setminus X$, and it is of dimension smaller than $d$ by Proposition \ref{prop:dim}(2),(5),(8).
	We have $\dim W=d$ by Proposition \ref{prop:dim}(5).
	Take a point $x \in W$ such that, for any open box $B$ containing the point $x$, the equality $ \dim B \cap W = d$ holds true by Proposition \ref{prop:dim}(10).
	By the definitions of $Z$ and $W$, if we take a sufficiently small open box $B$ containing the point $x$, the intersection $B \cap W$ has an empty intersection with $Z$.
	It implies that $X \cap B=Y\cap B$.
\end{proof}

We also need the fact that a definably complete locally o-minimal structure is definably Baire.
\begin{definition}
	Consider an expansion of a linearly ordered structure $\mathcal R=(R,<,0,\ldots)$.
	A \textit{parameterized family} of definable sets $\{X{\langle x \rangle}\}_{x \in S}$ is the family of the fibers of a definable set; that is, there exists a definable set $\mathcal X$ with $X{\langle x \rangle}=\mathcal X_x$ for all $x$ in a definable set $S$.
	
	A parameterized family of definable sets $\{X{\langle r\rangle}\}_{r>0}$  is called a \textit{definable increasing family} if $X{\langle r\rangle} \subseteq X{\langle r'\rangle}$ whenever $0<r<r'$.
	A definably complete expansion of a densely linearly ordered structure is \textit{definably Baire} if the union $\bigcup_{r>0} X{\langle r\rangle}$ of any definable increasing family $\{X{\langle r\rangle}\}_{r>0}$ with $\myint\left(\overline{X{\langle r\rangle}}\right)=\emptyset$ has an empty interior.
\end{definition}

\begin{proposition}\label{prop:baire}
	A definably complete locally o-minimal structure is definably Baire.
\end{proposition}
\begin{proof}
	It immediately follows from \cite[Proposition 2.9]{Fuji4} and \cite[Theorem 2.5]{FKK}.
\end{proof}

\begin{proposition}\label{prop:baire2}
	Consider a definably complete locally o-minimal structure.
	Let $\{X{\langle r\rangle}\}_{r>0}$ be a definable increasing family.
	We have $$\dim \left(\bigcup_{r>0} X{\langle r\rangle}\right)=\sup\{\dim(X{\langle r\rangle})\;|\;r>0\}\text{.}$$
\end{proposition}
\begin{proof}
	Let $\mathcal M=(M,<,+,0,\ldots)$ be the given definably complete locally o-minimal structure.
	Set $d=\sup\{\dim(X{\langle r\rangle})\;|\;r>0\}$ and $Y=\bigcup_{r>0} X{\langle r\rangle}$.
	We lead to a contradiction assuming that $\dim (Y)>d$.
	
	There exist a definable open box $B$ in $M^{d+1}$ and a definable continuous injective map $\varphi:U \rightarrow Y$ which is homeomorphic onto its image by \cite[Proposition 2.8(9)]{FKK}.
	Set $V{\langle r\rangle}=\varphi^{-1}(X{\langle r\rangle})$.
	We have $B=\bigcup_{r>0} V{\langle r\rangle}$ and $\dim V{\langle r\rangle} \leq d$ for all $r>0$ by Proposition \ref{prop:dim}(6).
	We also get $\dim \overline{V{\langle r\rangle}} \leq d$ by Proposition \ref{prop:dim}(5),(8).
	The definable sets $\overline{V{\langle r\rangle}}$ have empty interiors for all $r>0$ by the definition of dimension.
	Since $B=\bigcup_{r>0} V{\langle r\rangle}$, the open box $B$ has an empty interior by Proposition \ref{prop:baire}.
	Contradiction. 
\end{proof}

We next prove a definable choice lemma.
\begin{lemma}[Definable choice lemma]\label{lem:definable_choice}
Consider a definably complete expansion of an ordered group $\mathcal M=(M,<,0,+\ldots)$.
Let $\pi:M^{m+n} \rightarrow M^m$ be a coordinate projection.
Let $X$ and $Y$ be definable subsets of $M^m$ and $M^{m+n}$, respectively,  satisfying the equality $\pi(Y)=X$.
There exists a definable map $\varphi:X \rightarrow Y$ such that $\pi(\varphi(x))=x$ for all $x \in X$.
\end{lemma}
\begin{proof}
We may assume that $\pi$ is the coordinate projection onto the first $m$ coordinate without loss of generality.
We fix a positive element $c \in M$.
We prove the lemma by induction on $m$.

When $m=1$, we define $\varphi:X \rightarrow Y$ as follows:
Fix a point $x \in X$.
Consider the fiber $Y_x=\{y \in M\;|\; (x,y) \in Y\}$.
Set $Y_x^+=\{y \in Y_x\;|\; y \geq 0\}$ and $Y_x^-=\{y \in Y_x\;|\; y \leq 0\}$.
When the definable set $Y_x^+$ is not an empty set, consider the element $y_1=\inf(Y_x^+)$.
If $y_1 \in Y_x$, set $\varphi(x)=(x,y_1)$.
Otherwise, consider $y_2=\sup\{y > y_1\;|\;y' \in Y_x^+ \text{ for all }y_1<y'<y\} \in M \cup \{\infty\}$.
When $y_2=\infty$, put $\varphi(x)=(x,y_1+c)$.
Otherwise, set $\varphi(x)=(x,(y_1+y_2)/2)$.
We can define $\varphi:X \rightarrow Y$ in the same manner considering $Y_x^-$ instead of $Y_x^+$ when $Y_x^+$ is an empty set.

We next consider the case in which $n>1$.
Let $\pi_1:M^{m+n} \rightarrow M^{m+n-1}$ and $\pi_2:M^{m+n-1} \rightarrow M^{m}$ be the coordinate projections forgetting the last coordinate and forgetting the last $n-1$ coordinates, respectively.
Set $Z=\pi_1(Y)$.
We have $\pi_2(Z)=X$.
Applying the induction hypothesis to the pair of $X$ and $Z$, we get a definable map $\varphi_2:X \rightarrow Z$ such that the composition $\pi_2 \circ \varphi_2$ is the identity map.
Applying the lemma for $n=1$ to the pair of $Z$ and $Y$, we obtain a definable map $\varphi_1:Z \rightarrow Y$ with $\pi_1(\varphi_1(z))=z$ for all $z \in Z$.
The composition $\varphi=\varphi_1 \circ \varphi_2$ is the definable map we are looking for.
\end{proof}

The following curve selection lemma is worth to be mentioned.

\begin{corollary}\label{cor:curve_selection}
Consider a definably complete locally o-minimal expansion of an ordered group $\mathcal M=(M,<,+,0,\ldots)$.
Let $X$ be a definable subset of $M^n$ which is not closed.
Take a point $a \in \mycl(X) \setminus X$.
There exist a small positive $\varepsilon$ and a definable continuous map $\gamma:(0,\varepsilon) \rightarrow X$ such that $\lim_{t \to +0}\gamma(t)=a$.
\end{corollary}
\begin{proof}
Let $\pi:M^{n+1} \rightarrow M$ be the projection onto the last coordinate.
Set $Y=\{(x,t) \in X \times M\;|\; |a-x|=t\}$.
Since $\mathcal M$ is locally o-minimal, the intersection $(-\delta, \delta) \cap \pi(Y)$ is a finite union of points and open intervals for a sufficiently small $\delta>0$.
Since the point $a$ belongs to the closure of $X$, the intersection $(-\delta, \delta) \cap \pi(Y)$ contains an open interval of the form $(0,\varepsilon)$ for some $\varepsilon>0$.
There exists a definable map $\gamma:(0,\varepsilon) \rightarrow X$ with $(\gamma(t),t) \in Y$ for all $0 < t < \varepsilon$ by Lemma \ref{lem:definable_choice}.
It is obvious that the map $\gamma$ is bounded.
Taking a smaller $\varepsilon>0$ if necessary, we may assume that $\gamma$ is continuous by Proposition \ref{prop:limit}(1).
The equality $\lim_{t \to +0}\gamma(t)=a$ is obvious by the definition of $\gamma$.
\end{proof}

We also use the following lemma:
\begin{lemma}\label{lem:proj_dim}
	Consider a definably complete locally o-minimal expansion of an ordered group $\mathcal M=(M,<,+,0,\ldots)$.
	Let $C$ and $P$ be definable subsets of $M^m$ and $M^n$, respectively.
	Let $X$ be a definable subset of $C \times P$.
	Let $\pi:M^{m+n} \rightarrow M^n$ denotes the projection onto the last $n$ coordinates.
	Assume that $\dim \pi(X)=\dim P$.
	Then there exists a point $(c,p) \in X$ such that $\dim \pi(X \cap W) = \dim P$ for all open boxes $W$ in $M^{m+n}$ containing the point $(c,p)$.
\end{lemma}
\begin{proof}
	We can find a definable map $\tau:\pi(X) \rightarrow X$ such that the composition $\pi \circ \tau$ is the identity map on $\pi(X)$ by Lemma \ref{lem:definable_choice}.
	Let $D$ be the closure of the set of points at which $\tau$ is discontinuous.
	We have $\dim D < \dim \pi(X)=\dim P$ by Proposition \ref{prop:dim}(5),(7),(8).
	Set $E=\pi(X) \setminus D$.
	We obtain $\dim E=\dim P$ by Proposition \ref{prop:dim}(5).
	Therefore there exists a point $p \in E$ with $\dim (E \cap U)=\dim P$ for all open box $U$ in $M^n$ containing the point $p$ by Proposition \ref{prop:dim}(10).
	Set $(c,p)=\tau(p)$.
	
	We demonstrate that the point $(c,p)$ satisfies the condition in the lemma.
	Take an arbitrary sufficiently small open box $W$ in $M^{m+n}$ containing the point $(c,p)$.
	We may assume that $D \cap \pi(W) = \emptyset$ because $p \not\in D$ and $D$ is closed.
	Since $\tau$ is continuous on $E$, the set $\tau^{-1}(W)=\pi(\tau(E) \cap W)$ is open in $E$.
	There exists an open box $U$ in $R^n$ such that $p \in U$ and $E \cap U \subset \pi(\tau(E) \cap W)$.
	Shrinking $U$ if necessary, we may assume that $U$ is contained in $\pi(W)$.
	We have $\dim P=\dim E \cap U$ by the definition of the point $p$.
	We then get $\dim P=\dim E \cap U \leq \dim \pi(\tau(E) \cap W) \leq \dim \pi(X \cap W) \leq \dim \pi(X) \leq \dim P$ by Proposition \ref{prop:dim}(2).
	We have demonstrated the lemma.
\end{proof}

\begin{definition}
Consider a definably complete expansion of an ordered group $\mathcal M=(M,<,0,+\ldots)$.
Let $X$ and $Y$ be definable subsets of $M^n$.
The \textit{distance} of $X$ to $Y$ is given by $\inf\{|x-y|\;|\;x \in X,\ y \in Y\}$.
\end{definition}

\begin{lemma}\label{lem:dist}
Consider a definably complete locally o-minimal expansion of an ordered group $\mathcal M=(M,<,0,+\ldots)$.
Let $X$ and $Y$ be mutually disjoint definable, closed and bounded subsets of $M^n$.
Then, the distance of $X$ to $Y$ is positive.
\end{lemma}
\begin{proof}
We prove the contraposition of the lemma.
Assume that the distance of $X$ to $Y$ is zero.
We prove that the intersection $X \cap Y$ is not empty when $X$ and $Y$ are closed and bounded.
By local o-minimality and the assumption, the set $$Z=\{|x-y|\;|\; x \in X,\ y \in Y\}$$ contains the origin or an open interval of the form $(0,\delta)$, where $\delta$ is a positive element.
It is obvious that the  intersection $X \cap Y$ is not empty when $Z$ contains the origin. 
We concentrate on the case in which $Z$ contains the open interval $(0,\delta)$ in the rest of the proof.
Applying Lemma \ref{lem:definable_choice} to $Z$, we can choose definable maps $x,y:(0,\delta) \rightarrow M^n$ such that, for any $0<t<\delta$, $x(t) \in X$, $y(t) \in Y$ and $|x(t)-y(t)|=t$.
These maps are bounded because $X$ and $Y$ are bounded.
We may assume that they are continuous by taking a smaller $\delta$ if necessary by Proposition \ref{prop:limit}(1).
Set $a=\lim_{t \to 0}x(t)$ and $b=\lim_{t \to 0}y(t)$, which exist by Proposition \ref{prop:limit}(2).
Since both $X$ and $Y$ are closed, and both the maps $x$ and $y$ are continuous, we have $a \in X$, $b \in Y$ and $|a-b|=0$.
It means that the point $a(=b)$ is a common point of $X$ and $Y$.
\end{proof}

\section{Functions definable in a definably complete o-minimal expansion of an ordered group}\label{sec:function}
We prove several properties enjoyed by definable functions in this section.
We first prove a technical lemma.
\begin{lemma}\label{lem:main}
Consider a definably complete locally o-minimal expansion of an ordered group $\mathcal M=(M,<,+,0,\ldots)$.
Let $C$ be a definable, closed and bounded subset of $R^m$.
Let $\varphi, \psi: C \rightarrow M_{>0}$ be two definable functions.
Assume that the following condition is satisfied:
$$
\forall x \in C, \exists \delta>0, \forall x' \in C,\ |x'-x| < \delta \Rightarrow \varphi(x') \geq \psi(x) \text{.}
$$
Then the inequality $\inf \varphi(C)>0$ holds true.
\end{lemma}
\begin{proof}
Set $l=\inf\varphi(C) \geq 0$, which exists by the definable completeness of $\mathcal M$.
We have only to show that $l>0$.
Since $\mathcal M$ is locally o-minimal, we have $l \in \varphi(C)$ or there exists $u \in M$ with $l < u$ and $(l,u) \subseteq \varphi(C)$.
It is obvious that $l>0$ in the former case.
We consider the latter case in the rest of the proof.

Let $\Gamma$ be the graph of the function $\varphi$.
Let $\pi_1:M^{m+1} \rightarrow M^m$ and $\pi_2:M^{m+1} \rightarrow M$ be the projections onto the first $m$ coordinates and onto the last coordinate, respectively.
We can take a definable map $\eta:(l,u) \rightarrow \Gamma$ such that the composition $\pi_2 \circ \eta$ is the identity map on $(l,u)$ by Lemma \ref{lem:definable_choice}.
Note that the map $\eta$ is bounded because the domain of definition $C$ of $\varphi$ is bounded and the interval $(l,u)$ is bounded.
By Proposition \ref{prop:limit}(1), we may assume that $\eta$ is continuous by taking a smaller $u$ if necessary.

Set $z=\lim_{t \to l+}\eta(t)$, which uniquely exists by Proposition \ref{prop:limit}(2).
We have $\pi_2(z)=l$ by the definition of $\eta$.
Set $c=\pi_1(z)$.
It belongs to $C$ because $C$ is bounded and closed.
For any $t>l$ sufficiently close to $l$, $\pi_1(\eta(t)) \in C$ is close to the point $c$.
We have $\pi_2(\eta(t)) =\varphi(\pi_1(\eta(t))) \geq \psi(c)$ for such $t$ by the assumption.
We finally obtain $l=\lim_{t \to l+}\pi_2(\eta(t))) \geq \psi(c)>0$.
\end{proof}

We investigate the properties of functions definable in a definably complete o-minimal expansion of an ordered group.
\begin{definition}
Consider an expansion of a densely linearly ordered abelian group $\mathcal M=(M,<,+,0,\ldots)$.
Let $C$ and $P$ be definable sets.
Let $f: C \times P \rightarrow M$ be a definable function.
The function $f$ is \textit{equi-continuous} with respect to $P$ if the following condition is satisfied:
$$
\forall \varepsilon>0, \ \forall x \in C, \ \exists \delta >0, \ \forall p \in P, \ \forall x' \in C,\ \  |x-x'|< \delta \Rightarrow |f(x,p)-f(x',p)|< \varepsilon\text{.}
$$
The function $f$ is \textit{uniformly equi-continuous} with respect to $P$ if the following condition is satisfied:
$$
\forall \varepsilon>0, \ \exists \delta >0, \ \forall p \in P, \ \forall x, x' \in C,\ \  |x-x'|< \delta \Rightarrow |f(x,p)-f(x',p)|< \varepsilon\text{.}
$$

The function $f$ is \textit{pointwise bounded} with respect to $P$ if the following condition is satisfied:
$$
\forall x \in C,\  \exists N>0, \  \forall p \in P,\ \   |f(x,p)|<N\text{.}
$$
\end{definition}

\begin{proposition}\label{prop:equi-cont}
Consider a definably complete locally o-minimal expansion of an ordered group $\mathcal M=(M,<,+,0,\ldots)$.
Let $C$ and $P$ be definable sets.
Let $f: C \times P \rightarrow M$ be a definable function.
Assume that $C$ is closed and bounded.
Then $f$ is equi-continuous with respect to $P$ if and only if it is uniformly equi-continuous with respect to $P$.
\end{proposition}
\begin{proof}
A uniformly equi-continuous definable function is always equi-continuous.
We prove the opposite implication.

Take a positive $c \in M$.
Consider the definable function $\varphi:C \times M_{>0} \rightarrow M_{>0}$ given by 
$$ \varphi(x,\varepsilon)=\sup\{0<\delta<c\;|\;\forall p \in P, \ \forall x' \in C,\  |x-x'|< \delta \Rightarrow |f(x,p)-f(x',p)|< \varepsilon\}\text{.}$$
Since $f$ is equi-continuous with respect to $P$, we have $\varphi(x,\varepsilon)>0$ for all $x \in C$ and $\varepsilon>0$.
Fix arbitrary $x \in C$ and $\varepsilon>0$.
We also fix an arbitrary point $x' \in C$ with $|x'-x|<\frac{1}{2}\varphi(x,\frac{\varepsilon}{2})$.
We have $|f(x',p)-f(x,p)|<\frac{\varepsilon}{2}$ for all $p \in P$ by the definition of $\varphi$.

For all $y \in C$ with $|x'-y|<\frac{1}{2}\varphi(x,\frac{\varepsilon}{2})$, we have $|x-y| \leq |x-x'|+|x'-y| < \varphi(x,\frac{\varepsilon}{2})$.
We get $|f(y,p)-f(x,p)|<\frac{\varepsilon}{2}$ for all $p \in P$ by the definition of $\varphi$.
We finally obtain $|f(y,p)-f(x',p)| \leq |f(x',p)-f(x,p)|+|f(y,p)-f(x,p)|<\varepsilon$ for all $p \in P$.
It means that $\varphi(x',\varepsilon) \geq \frac{1}{2}\varphi(x,\frac{\varepsilon}{2})$ whenever $|x'-x|<\frac{1}{2}\varphi(x,\frac{\varepsilon}{2})$.
Apply Lemma \ref{lem:main} to the definable functions $\varphi(\cdot,\varepsilon)$ and $\frac{1}{2}\varphi(\cdot,\frac{\varepsilon}{2})$ for a fixed $\varepsilon>0$.
We have $\inf \varphi(C,\varepsilon)>0$.

For any $\varepsilon>0$, set $\delta=\inf \varphi(C,\varepsilon)$.
For any $p \in P$ and $x,x' \in C$, we have $|f(x,p)-f(x',p)|< \varepsilon$ whenever $|x-x'|< \delta$ by the definition of $\varphi$.
It means that $f$ is uniformly equi-continuous.
\end{proof}

It is well-known that a continuous function defined on a compact set is uniformly continuous.
The following corollary claims that a similar assertion holds true for a definable function defined on a definable, closed bounded set.
\begin{corollary}\label{cor:uniform}
Consider a definably complete locally o-minimal expansion of an ordered group $\mathcal M=(M,<,+,0,\ldots)$.
Let $C$ be a definable, closed and bounded set.
A definable continuous function $f: C \rightarrow M$ is uniformly continuous.
\end{corollary}
\begin{proof}
Let $P$ be a singleton.
Apply Proposition \ref{prop:equi-cont} to the function $g:C \times P \rightarrow M$ defined by $g(x,p)=f(x)$.
\end{proof}

We define a definable family of functions and investigate its properties.
Equi-continuity, convergence and uniform convergence are defined for sequences of functions in classical analysis.
We consider similar notions for a definable family of functions.
\begin{definition}
Consider an expansion of a densely linearly ordered abelian group $\mathcal M=(M,<,+,0,\ldots)$.
Let $C$ be a definable set and $s$ be a positive element in $M$.
A family $\{f_t:C \rightarrow M\}_{0<t<s}$ of functions with the parameter variable $t$ is a \textit{definable family of functions} if there exists a definable function $F:C \times (0,s) \rightarrow M$ such that $f_t(x)=F(x,t)$ for all $x \in C$ and $0<t<s$.
We call it a \textit{definable family of continuous functions} if every function $f_t$ is continuous.

Consider a definable family of functions $\{f_t:C \rightarrow M\}_{0<t<s}$.
Set $I=(0,s)$.
The map $F:C \times I \rightarrow M$ given by $F(x,t) =f_t(x)$ is a definable function by the definition.
The family is a \textit{definable family of equi-continuous functions} if $F$ is equi-continuous with respect to $I$.
It is a \textit{definable family of pointwise bounded functions} if $F$ is pointwise bounded with respect to $I$.

A definable family of functions $\{f_t:C \rightarrow M\}_{0<t<s}$ is \textit{pointwise convergent} if for any positive $\varepsilon>0$ and for any $x \in C$, there exists $s'>0$ such that $|f_t(x)-f_{t'}(x)|<\varepsilon$ for all $t,t' \in (0,s')$.
\end{definition}

The following lemma is proved following a typical argument in classical analysis.
\begin{lemma}\label{lem:pointwise}
Consider an expansion of a densely linearly ordered abelian group $\mathcal M=(M,<,+,0,\ldots)$.
Let $C$ be a definable set and $s$ be a positive element in $M$.
Consider a pointwise convergent definable family of functions $\{f_t:C \rightarrow M\}_{0<t<s}$.
For any $x \in C$, there exists $s'>0$ such that the set $\{f_t(x)\;|\;0<t<s'\}$ is bounded.
\end{lemma}
\begin{proof}
Fix $x \in C$.
Take a positive $\varepsilon >0$.
There exists $s'>0$ such that$|f_t(x)-f_{t'}(x)|<\varepsilon$ for all $t,t' \in (0,s')$.
Fix $u \in (0,s')$.
For any $t \in (0,s')$, we have $$|f_t(x)| \leq |f_u(x)|+|f_u(x)-f_t(x)| < |f_u(x)|+ \varepsilon.$$
It means that the set $\{f_t(x)\;|\;0<t<s'\}$ is bounded.
\end{proof}

We also get the following converse when $\mathcal M$ is a a definably complete locally o-minimal expansion of a densely linearly ordered abelian group.
\begin{lemma}\label{lem:pointwise2}
Consider a definably complete locally o-minimal expansion of an ordered group $\mathcal M=(M,<,+,0,\ldots)$.
Let $C$ be a definable set and $s$ be a positive element in $M$.
A definable family of pointwise bounded functions $\{f_t:C \rightarrow M\}_{0<t<s}$ is pointwise convergent.
\end{lemma}
\begin{proof}
Fix $x \in C$.
Set $I=(0,s)$.
Consider the definable function $g:I \rightarrow M$ given by $g(t)=f_t(x)$.
It is bounded.
There exists a limit $y=\lim_{t \to +0}g(t)$ by Proposition \ref{prop:limit}(2).

Take an arbitrary positive $\varepsilon >0$.
There exists $s'>0$ such that$|y-g(t)|<\varepsilon/2$ for all $t\in (0,s')$.
We have $|f_t(x)-f_{t'}(x)| \leq |f_t(x)-y|+|y-f_{t'}(x)|<\varepsilon$ whenever $t, t'\in (0,s')$.
It means that the family $\{f_t:C \rightarrow M\}_{0<t<s}$ is pointwise convergent.
\end{proof}

We define the limit of a pointwise convergent definable family of functions.
\begin{definition}\label{def:lim_func}
Consider an expansion of a densely linearly ordered abelian group $\mathcal M=(M,<,+,0,\ldots)$.
Let $C$ be a definable set and $s$ be a positive element in $M$.
Consider a pointwise convergent definable family of functions $\{f_t:C \rightarrow M\}_{0<t<s}$.
For any $x \in C$, consider the function $g_x:(0,s) \rightarrow M$ given by $g_x(t)=f_t(x)$.
Taking a smaller $s>0$ if necessary, we may assume that $g_x$ is bounded by Lemma \ref{lem:pointwise}.
There exists a unique limit $\lim_{t \to +0} g_x(t)$ exists by Proposition \ref{prop:limit}(2).
The \textit{limit} $\lim_{t \to +0}f_t:C \rightarrow M$ of the family $\{f_t:C \rightarrow M\}_{0<t<s}$ is defined by $(\lim_{t \to +0}f_t)(x)=\lim_{t \to +0}g_x(t)$.
\end{definition}

\begin{definition}
Consider an expansion of a densely linearly ordered abelian group $\mathcal M=(M,<,+,0,\ldots)$.
Let $C$ be a definable set and $s$ be a positive element in $M$.
A definable family of functions $\{f_t:C \rightarrow M\}_{0<t<s}$ is \textit{uniformly convergent} if for any positive $\varepsilon>0$, there exists $s'>0$ such that $|f_t(x)-f_{t'}(x)|<\varepsilon$ for all $x \in C$ and $t,t' \in (0,s')$.
\end{definition}

The following proposition and its proof is almost the same as the counterparts in classical analysis.
\begin{proposition}\label{prop:uniform_convergent}
Consider a definably complete locally o-minimal expansion of an ordered group $\mathcal M=(M,<,+,0,\ldots)$.
Let $C$ be a definable set and $s$ be a positive element in $M$.
Consider a uniformly convergent definable family of continuous functions $\{f_t:C \rightarrow M\}_{0<t<s}$.
The limit $\lim_{t \to +0}f_t:C \rightarrow M$ is continuous.
\end{proposition}
\begin{proof}
Fix arbitrary $\varepsilon>0$ and $x \in C$.
Since the family is uniformly convergent, we may assume that $|f_t(x')-f_{t'}(x')|<\frac{\varepsilon}{5}$ for all $x' \in C$ and $t,t' \in (0,s)$ by taking a smaller $s>0$ if necessary.
Fix $t_0$ with $0<t_0<s$.
There exists $\delta>0$ such that $|f_{t_0}(x')-f_{t_0}(x)|<\frac{\varepsilon}{5}$ whenever $|x-x'|<\delta$ because $f_{t_0}$ is continuous.
Fix an arbitrary point $x' \in C$ with $|x-x'|<\delta$.
We can take $t_1,t_2 \in (0,s)$ with $|(\lim_{t \to +0}f_t)(x)-f_{t_1}(x)|<\frac{\varepsilon}{5}$ and $|(\lim_{t \to +0}f_t)(x')-f_{t_2}(x')|<\frac{\varepsilon}{5}$ by the definition of the limit $\lim_{t \to +0}f_t$.
We finally have 
\begin{align*}
&|(\lim_{t \to +0}f_t)(x')-(\lim_{t \to +0}f_t)(x)| \\
&\leq |(\lim_{t \to +0}f_t)(x')-f_{t_2}(x')|+|f_{t_2}(x')-f_{t_0}(x')|+|f_{t_0}(x')-f_{t_0}(x)|\\
&\quad +|f_{t_0}(x)-f_{t_1}(x)|+|f_{t_1}(x)-(\lim_{t \to +0}f_t)(x)|\\
&<\varepsilon.
\end{align*}
We have proven that $\lim_{t \to +0}f_t$ is continuous.
\end{proof}

The following Arzela-Ascoli-type theorem is a main theorem of this paper.

\begin{theorem}\label{thm:ascoli}
Consider a definably complete locally o-minimal expansion of an ordered group $\mathcal M=(M,<,+,0,\ldots)$.
Let $C$ be a definable, closed and bounded set.
A pointwise convergent definable family of equi-continuous functions $\{f_t:C \rightarrow M\}_{0<t<s}$ is uniformly convergent. 
\end{theorem}
\begin{proof}
Set $I=(0,s)$.
Consider the map $F:C \times I \rightarrow M$ given by $F(x,t) =f_t(x)$.
It is an equi-continuous definable function with respect to $I$ by the definition.
Set $g=\lim_{t \to +0} f_t$.
It is well-defined by Definition \ref{def:lim_func} because the family is pointwise convergent.

Take $c>0$.
Consider the definable function $\varphi:C \times M_{>0} \rightarrow M_{>0}$ given by
$$\varphi(x,\varepsilon)=\sup\{0<\delta<c\;|\; \forall t,t' \in (0,\delta),\ |F(x,t)-F(x,t')|<\varepsilon\}\text{.}$$
 We first show that it is well-defined.
 Fix $x \in C$ and $\varepsilon>0$.
 There exists $\delta>0$ such that $|F(x,u)-g(x)|<\frac{\varepsilon}{2}$ for all $u \in (0,\delta)$ by the definition of $g$.
For any $t,t' \in  (0,\delta)$, we have $|F(x,t)-F(x,t')| \leq |F(x,t)-g(x)|+|g(x)-F(x,t')|<\varepsilon$.
The definable set $\{0<\delta<c\;|\; \forall t,t' \in (0,\delta),\ |F(x,t)-F(x,t')|<\varepsilon\}$ is not empty and the function $\varphi$ is well-defined.

We fix $x \in C$ and $\varepsilon>0$ again.
Since $F$ is equi-continuous with respect to $I$, there exists $\delta'>0$ such that 
$$\forall t \in (0,s),\ \forall x' \in C, \ |x-x'|< \delta' \Rightarrow |F(x,t)-F(x',t)|<\frac{\varepsilon}{3}\text{.}$$
Fix an arbitrary $x' \in C$ with $|x-x'|< \delta'$.
For any $t,t' \in (0,\varphi(x, \frac{\varepsilon}{3}))$, we have $|F(x,t)-F(x,t')|<\frac{\varepsilon}{3}$ by the definition of $\varphi$.
We finally get $$|F(x',t)-F(x',t')| \leq |F(x',t)-F(x,t)|+|F(x,t)-F(x,t')|+|F(x,t')-F(x',t')|<\varepsilon$$
whenever $t,t' \in (0,\varphi(x, \frac{\varepsilon}{3}))$.
It means that $\varphi(x',\varepsilon) \geq \varphi(x,\frac{\varepsilon}{3})$.
Apply Lemma \ref{lem:main} to the definable functions $\varphi(\cdot,\varepsilon)$ and $\varphi(\cdot,\frac{\varepsilon}{3})$ for a fixed $\varepsilon>0$.
We have $\inf \varphi(C,\varepsilon)>0$ for all $\varepsilon>0$.

Fix $\varepsilon>0$.
Set $\delta= \inf \varphi(C,\varepsilon)>0$.
We have $|f_t(x)-f_{t'}(x)|=|F(t,x)-F(t',x)|<\varepsilon$ for all $x \in C$ and $t,t' \in (0,\delta)$.
It means that the family $\{f_t:C \rightarrow M\}_{0<t<s}$ is uniformly convergent. 
\end{proof}

The above theorem together with the curve selection lemma yields the following corollary:
\begin{corollary}\label{cor:ascoli}
Consider a definably complete locally o-minimal expansion of an ordered group $\mathcal M=(M,<,+,0,\ldots)$.
Let $C$ and $P$ be definable sets.
Assume that $C$ is closed and bounded.
Let $f:C \times P \rightarrow M$ be a definable function which is equi-continuous and pointwise bounded with respect to $P$.
Take $p \in \mycl(P)$.
There exists a definable continuous curve $\gamma:(0,\varepsilon) \rightarrow P$ such that $\lim_{t \to +0}\gamma(t)=p$ and the definable family of functions $\{g_t:C \rightarrow M\}_{0<t<\varepsilon}$ defined by $g_t(x)=f(x,\gamma(t))$ is uniformly convergent.
\end{corollary}
\begin{proof}
The corollary follows from Corollary \ref{cor:curve_selection}, Lemma \ref{lem:pointwise2} and Theorem \ref{thm:ascoli}.
\end{proof}

Consider a parameterized function $f:C \times P \rightarrow M$ which is equi-continuous with respect to $P$.
We show that the projection image of the set at which $f$ is discontinuous onto the parameter space $P$ is of dimension smaller than $\dim P$ when $C$ is closed.

\begin{theorem}\label{thm:discont}
	Consider a definably complete expansion of an ordered group $\mathcal M=(M,<,+,0,\ldots)$.
	Let $C$ be a definable closed set and $P$ be a definable set.
	Let $\pi:C \times P \rightarrow P$ be the projection.
	Consider a definable function $f:C \times P \rightarrow M$ which is equi-continuous with respect to $P$.
	Set $D= \{(x,q) \in C \times P\;|\; f \text{ is discontinuous at }(x,q)\}$.
	We have $\dim \pi(D) < \dim P$.
\end{theorem}
\begin{proof}
	Let $C$ and $P$ be definable subsets of $M^m$ and $M^n$, respectively.
	
	We first consider the case in which $C$ is bounded.
	Consider the set
	$$
	S=\{(x,p) \in D\;|\; \exists U \subset M^m:\text{open box with }x \in U \text{ and }C \cap U = D_p \cap U\}\text{,}$$
	where the notation $D_p$ denotes the fiber $\{x \in C\;|\; (x,p) \in D\}$.
	We demonstrate that $\dim \pi(S) < \dim P$.
	
	Assume for contradiction that $\dim \pi(S)=\dim P$.
	We can take an open box $B$ in $M^m$ such that $\dim \pi(S) \cap B=\dim P$ and $\pi(S) \cap B=P \cap B$ by Corollary \ref{cor:point}.
	Using Lemma \ref{lem:definable_choice}, we can construct a definable map $\tau:P \cap B \rightarrow S$ such that the composition $\pi \circ \tau$ is the identity map on $P \cap B$.
	The set of points at which $\tau$ is discontinuous is of dimension smaller than $\dim P$ by Proposition \ref{prop:dim}(7).
	By shrinking $B$ if necessary, we may assume that $\tau$ is continuous by  Proposition \ref{prop:dim}(5) and Corollary \ref{cor:point}.
	Fix a positive $K>0$.
	Consider the definable map $\varphi:P \cap B \rightarrow M$ given by $$\varphi(x)=\sup\{0<\lambda<K\;|\; C \cap \mathcal B_m(\tau(p),\lambda) = D_p \cap \mathcal B_m(\tau(p),\lambda)\}.$$
	We may assume that $\varphi$ is continuous in the same manner as above.
	Set $$W=\bigcup_{p \in P \cap B} (\mathcal B_m(\tau(p),\varphi(p)) \cap C)\times \{p\}.$$
	It is obvious that $W$ is an open subset of $C \times P$ and $W \subseteq D$.
	Since $W$ is open in $C \times P$ and $W \subseteq D$, the restriction of $f$ to $W$ is discontinuous everywhere.
	It contradicts Proposition \ref{prop:dim}(7).
	We have demonstrated the inequality $\dim\pi(S)<\dim P$.
	
	We next demonstrate that $\dim \pi(D)<\dim P$.
	We lead to a contradiction assuming the contrary.
	Set $T=D \setminus \pi^{-1}(\pi(S))$.
	We have $\dim \pi(T)= \dim P$ by Proposition \ref{prop:dim}(5) because $\dim\pi(S)<\dim P$.
	There exists a point $(c,p) \in T$ such that $\dim \pi(T \cap W)=\dim P$ for all open box $W$ in $M^{m+n}$ containing the point $(c,p)$ by Lemma \ref{lem:proj_dim}.
	Fix an arbitrary $\varepsilon>0$.
	Since $f$ is uniformly equi-continuous with respect to $P$ by the assumption and Proposition \ref{prop:equi-cont}, there exists $\delta>0$ satisfying the following condition:
	\begin{equation}\label{eq:equi-cont1}
	\forall q \in P, \ \forall x,x' \in C,\ |x-x'|<\delta \Rightarrow |f(x,q)-f(x',q)|<\varepsilon/3\text{.}
	\end{equation}
	On the other hand, we have $D_p \cap U \subsetneq C \cap U$ for any open box $U$ containing the point $c$ by the definition of the set $S$ because $(c,p) \notin S$.
	In particular, there exists $c_0 \in C$ such that $|c-c_0|<\delta/2$ and $(c_0,p) \not\in D$.
	It implies that the function $f$ is continuous at the point $(c_0,p)$.
	There exists $\delta'>0$ such that 
	\begin{equation}\label{eq:equi-cont2}
	\forall q \in P,\ |q-p|<\delta' \Rightarrow |f(c_0,q)-f(c_0,p)|<\varepsilon/3.
	\end{equation}
	
	Consider an arbitrary point $(c',p') \in C \times P$ with $|c-c'|<\delta/2$ and $|p-p'|<\delta'$.
	We have $|f(c_0,p)-f(c,p)|<\varepsilon/3$ by the inequality (\ref{eq:equi-cont1}) because $|c-c_0|<\delta/2$.
	We also have $|f(c',p')-f(c_0,p')|<\varepsilon/3$ by (\ref{eq:equi-cont1}) because $|c'-c_0| \leq |c'-c|+|c-c_0|<\delta$.
	We get 
	\begin{align*}
	|f(c',p')-f(c,p)| &\leq |f(c',p')-f(c_0,p')|+|f(c_0,p')-f(c_0,p)|+|f(c_0,p)-f(c,p)|\\
	&<\varepsilon
	\end{align*}
	by the above inequalities together with the inequality (\ref{eq:equi-cont2}).
	We have demonstrated that $f$ is continuous at $(c,p)$.
	It is a contradiction to the condition that $(c,p) \in T \subset D$.
	We have demonstrated the theorem when $C$ is bounded.
	
	We next treat the general case.
	The definable closed set $C$ is not necessarily bounded.
	For any $r>0$, we set $B\langle r \rangle=[-r,r]^n$ and $C\langle r \rangle = B \langle r \rangle \cap C$.
	We also set $D\langle r \rangle= \{(x,q) \in C\langle r \rangle \times P\;|\; f|_{C\langle r \rangle \times P} \text{ is discontinuous at }(x,q)\}$ for all $r>0$.
	Here, the notation $f|_{C\langle r \rangle \times P}$ denote the restriction of $f$ to $C\langle r \rangle \times P$.
	We obviously have $D=\bigcup_{r>0}D\langle r \rangle$ and $\pi(D)=\bigcup_{r>0}\pi(D\langle r \rangle)$.
	The family $\{\pi(D\langle r \rangle)\}_{r>0}$ is a definable increasing family.
	Since the theorem holds true when $C$ is bounded, we have $\dim \pi(D\langle r \rangle)<\dim P$ for all $r>0$.
	We get $\dim \pi(D)<\dim P$ by Proposition \ref{prop:baire2}.
\end{proof}

\section{Decomposition into special submanifolds}\label{sec:decomposition}
\subsection{Definition of quasi-special/special submanifolds}
We first introduce our definition of special manifolds.
\begin{definition}\label{def:ours}
Consider an expansion of a dense linear order without endpoints $\mathcal M=(M,<,\ldots)$.
Let $\pi:M^n \rightarrow M^d$ be a coordinate projection, where $n$ is a positive integer and $d$ is a non-negative integer with $d \leq n$.
We consider that $M^0$ is a singleton equipped with the trivial topology and the projection $\pi:M^n \rightarrow M^0$ is the trivial map when $d=0$.
Let $\tau$ be the unique permutation of $\{1,\ldots, n\}$ such that 
\begin{enumerate}
\item[(a)] $\tau(i)<\tau(j)$ for $1 \leq i < j \leq n$ when $\tau(i)>d$ and $\tau(j)>d$.
\item[(b)] The composition $\pi \circ \overline{\tau}$ is the projection onto the first $d$ coordinates, where $\overline{\tau}:M^n \rightarrow M^n$ is the map defined by $\overline{\tau}(x_1,\ldots, x_n)=(x_{\tau(1)},\ldots, x_{\tau(n)})$.
\end{enumerate}
Set $$\myfib(X,\pi,x)=\{y \in M^{n-d}\;|\; (x,y) \in \overline{\tau}^{-1}(X)\}$$ for $x \in \pi(X)$.
Note that $\myfib(X,\pi,x)=\{y \in M^{n-d}\;|\; (x,y) \in X\}$ when $\pi$ is the projection onto the first $d$ coordinates.

When $\pi$ is the coordinate projection onto the first $d$ coordinate,
a definable subset $X$ of $M^n$ is a \textit{$\pi$-special submanifold} if, for any $x \in M^d$, there exist an open box $U$ in $M^d$ containing the point $x$ and a family $\{V_y\}_{y \in \myfib(X,\pi,x)}$ of mutually disjoint open boxes in $M^n$ indexed by the set $\myfib(X,\pi,x)$ such that 
\begin{enumerate}
\item[(1)] $\pi(V_y)=U$ for all $y \in \myfib(X,\pi,x)$;
\item[(2)] $X \cap \pi^{-1}(U)$ is contained in $\bigcup_{y \in \myfib(X,\pi,x)} V_y$, and 
\item[(3)] $V_y \cap X$ is the graph of a continuous map defined on $U$ for each $y \in \myfib(X,\pi,x)$.
\end{enumerate}
We do not require that the union $\bigcup_{y \in \myfib(X,\pi,x)} V_y$ is definable.

When $\pi$ is not the coordinate projection onto the first $d$ coordinate, 
we say that a definable subset $X$ of $M^n$ is \textit{$\pi$-special submanifold} if $\overline{\tau}^{-1}(X)$ is $\pi \circ \overline{\tau}$-special submanifold.
We omit the prefix $\pi$ when it is clear from the context.

Note that a discrete, closed definable subset of $M^n$ is always a $\pi$-special submanifold, where $\pi:M^n \rightarrow M^0$ is the trivial map.
\end{definition}

We also need the following definition:
\begin{definition}\label{def:normal}
Consider an expansion of a dense linear order without endpoints $\mathcal M=(M,<,\ldots)$.
Let $\pi:M^n \rightarrow M^d$ be a coordinate projection. 
Let $X$ be a definable subset of $M^n$.
Let $\tau$ be the permutation of $\{1,\ldots,n\}$ satisfying the conditions in Definition \ref{def:ours}.

A point $(a,b) \in M^n$ is \textit{$(X,\pi)$-normal} if there exist a definable neighborhood $A$ of $a$ in $M^d$ and a definable neighborhood $B$ of $b$ in $M^{n-d}$ such that either $A \times B$ is disjoint from $\overline{\tau}^{-1}(X)$ or $(A \times B) \cap \overline{\tau}^{-1}(X)$ is the graph of a definable continuous map $f:A \rightarrow B$.
\end{definition}

We get the following:
\begin{lemma}\label{lem:equal_size}
Let $\mathcal M=(M,<,0,+,\dots)$ be a definably complete locally o-minimal expansion of an ordered group.
Let $\pi:M^n \rightarrow M^d$ denote the projection onto the first $d$ coordinates.
Let $X$ be a definable subset of $M^n$.
Assume that any point in $X$ is $(X,\pi)$-normal.
Then, for any $x \in \pi(X)$ and sufficiently small positive $\varepsilon \in M$, there exists an open box $U$ in $M^d$ containing the point $x$ such that the intersection $X \cap (U \times \mathcal B_{n-d}(y,\varepsilon))$ is the graph of a continuous map for each $y \in \myfib(X,\pi,x)$.

In particular, when $X$ is a $\pi$-special submanifold, 
for any $x \in \pi(X)$ and sufficiently small positive $\varepsilon \in M$, there exists an open box $U$ in $M^d$ containing the point $x$ such that the pair $(U,\{U \times \mathcal B_{n-d}(y,\varepsilon)\}_{y \in \myfib(X,\pi,x)})$ satisfies the conditions (1) through (3) in Definition \ref{def:ours}.
\end{lemma}
\begin{proof}
We fix $x \in \pi(X)$.
Set $D=\myfib(X,\pi,x)$ for simplicity.
Fix a positive element $c \in M$.
We temporarily fix $y \in D$.
Because the point $(x,y)$ is $(X,\pi)$-normal, the set $(\mathcal B_d(x,\delta) \times \mathcal B_{n-d}(y,\varepsilon')) \cap X$ is the graph of a definable continuous map defined on $\mathcal B_d(x,\delta)$ when $\delta$  and $\varepsilon'$ are sufficiently small positive elements.
Consider the function $\sigma:D \rightarrow M$ given by 
\begin{align*}
\sigma(y)&=\sup\{0<\varepsilon'<c\;|\;\exists \delta>0\ \ (\mathcal B_d(x,\delta) \times \mathcal B_{n-d}(y,\varepsilon')) \cap X \text{ is }\\
&\qquad \text{the graph of a continuous map}\}.
\end{align*}
It is definable and the image $\sigma(D)$ is discrete and closed by Proposition \ref{prop:dim}(1),(6).
Set $\varepsilon_0=\inf \sigma(D)$.

Fix a sufficiently small $\varepsilon>0$ so that $\varepsilon<\varepsilon_0$.
Consider the function $\tau:D \rightarrow M$ given by 
\begin{align*}
\tau(y)&=\sup\{0<\delta<c\;|\;(\mathcal B_d(x,\delta) \times \mathcal B_{n-d}(y,\varepsilon)) \cap X \text{ is }\\
&\qquad \text{the graph of a continuous map}\}.
\end{align*}
The image $\tau(D)$ is discrete and closed for the same reason.
Set $\widetilde{\delta}=\inf \tau(D)$.
The open box $U=\mathcal B_d(x,\widetilde{\delta})$ satisfies the requirement of the lemma.

Let us consider the case in which $X$ is a $\pi$-special manifold.
Let $(U',\{V_y\}_{y \in D})$ be a pair satisfying the conditions (1) through (3) in Definition \ref{def:ours}.
It is easy to check that the pair $(U \cap U',\{(U \cap U') \times \mathcal B_{n-d}(y,\varepsilon)\}_{y \in D})$ satisfies the conditions (1) through (3) in Definition \ref{def:ours}.
We omit the details.
\end{proof}

We recall the definition of a quasi-special submanifold.

\begin{definition}\label{def:quasi-special}
Consider an expansion of a densely linearly order without endpoints $\mathcal M=(M,<,\ldots)$.
Let $\pi:M^n \rightarrow M^d$ be a coordinate projection. 
A definable subset is a \textit{$\pi$-quasi-special submanifold} or simply a \textit{quasi-special submanifold} if, for every point $x \in \pi(X)$, we can take an open box $U$ in $M^d$ containing the point $x$ and a family $\{V_y\}_{y \in \myfib(X,\pi,x)}$ of mutually disjoint open boxes in $M^n$ indexed by the set $\myfib(X,\pi,x)$ satisfying the conditions (1) and (3) in Definition \ref{def:ours}.
\end{definition}

It is obvious that a special submanifold is always a quasi-special submanifold.
The following example illustrates that the converse is false in general.
\begin{example}
Consider the ordered field of reals $(\mathbb R,<,0,1,+,\cdot)$.
The set $$\{(x,0)\;|\; x \in \mathbb R\} \cup \{(x,1/x)\;|\; x >0\}$$ is definable and a quasi-special submanifold, but it is not a special submanifold.
We can not take an open box $U$ and a family of open boxes $\{V_y\}_{y \in \myfib(X,\pi,x)}$ satisfying the condition (2) in Definition \ref{def:ours} at $x=0$.
\end{example}

We use the following lemma:
\begin{lemma}\label{lem:normal}
Consider a definably complete locally o-minimal structure $\mathcal M=(M,<,\ldots)$.
Let $X$ be a definable subset of $M^n$ and $\pi:M^n \rightarrow M^d$ be a coordinate projection.
Assume that all the points $x \in X$ are ($X,\pi$)-normal.
Then, $X$ is a $\pi$-quasi-special submanifold.
\end{lemma}
\begin{proof}
It immediately follows from \cite[Lemma 4.2]{Fuji4} and \cite[Theorem 2.5]{FKK}.
\end{proof}

We also need the following technical definition:
\begin{definition}
Consider an expansion of a densely linearly order without endpoints $\mathcal M=(M,<,\ldots)$.
Let $X$ be a definable subset of $M^n$ and $\pi:M^n \rightarrow M^d$ be a coordinate projection.
Let $x$ be a point in $\pi(X)$.
We say that $(X,\pi)$ is \textit{locally bounded at $x$} if there exist a bounded open box $U$ in $M^d$ containing the point $x$ such that $X \cap \pi^{-1}(U)$ is bounded.
\end{definition}

We give a sufficient condition for a quasi-special submanifold to be a special submanifold.
\begin{lemma}\label{lem:when}
Consider a definably complete locally o-minimal expansion of an ordered group $\mathcal M=(M,<,0,+,\ldots)$.
Let $\pi:M^n \rightarrow M^d$ be a coordinate projection.
A $\pi$-quasi-special submanifold $X$ in $M^n$ is a $\pi$-special submanifold if it is closed in $\pi^{-1}(\pi(X))$ and $(X,\pi)$ is locally bounded at every point in $\pi(X)$. 
\end{lemma}
\begin{proof}
We may assume that $\pi$ is the coordinate projection onto the first $d$ coordinates by permuting coordinates if necessary.
Fix an arbitrary point $x \in \pi(X)$.
Set $Y=\myfib(X,\pi,x)=\{y \in M^{n-d}\;|\;(x,y) \in X\}$.
Note that every point in $X$ is $(X,\pi)$-normal by the definition of quasi-special submanifolds.
Fix a sufficiently small positive $\varepsilon >0$.
We can take an open box $B$ containing the point $x$ such that $X \cap (B \times \mathcal B_{n-d}(y,\varepsilon))$ is the graph of a continuous map defined on $B$ for each $y \in Y$ by Lemma \ref{lem:equal_size}.
We may assume that $B$ is bounded and $X \cap (\mycl(B) \times \mathcal B_{n-d}(y,\varepsilon))$ is the graph of a continuous map defined on $\mycl(B)$ by shrinking $B$ if necessary.

Since $(X,\pi)$ is locally bounded at $x$, by shrinking $B$ again if necessary, $\pi^{-1}(B) \cap X$ is bounded.
Take a bounded open box $W$ so that $\pi^{-1}(B) \cap X$ is contained in $B \times W$.
Put $Z= \mycl(W) \setminus \bigcup_{y \in Y} \mathcal B_{n-d}(y,\varepsilon)$.
It is easy to show that the definable sets
\begin{align*}
C_1 &= \{x \} \times Y \text{ and }\\
C_2 &= X \cap (\mycl(B) \times Z)
\end{align*}
are closed and bounded, and their intersection is empty.
The proof is left to readers. 
Let $\mathfrak d$ be the distance of $C_1$ to $C_2$.
It is positive by Lemma \ref{lem:dist}.
Choose $\overline{\delta}>0$ so that $\overline{\delta}<\mathfrak d$ and $\mathcal B_d(x,\overline{\delta}) \subseteq B$.
The pair $(\mathcal B_d(x,\overline{\delta}),\{\mathcal B_d(x,\overline{\delta}) \times \mathcal B_{n-d}(y,\varepsilon)\}_{y \in Y})$ satisfies the conditions (1) through (3) in Definition \ref{def:ours}.
We have proven that $X$ is a $\pi$-special submanifold.
\end{proof}

\subsection{Comparison with other definitions}
We demonstrate that a special manifold defined in \cite{M2} and a multi-cell in \cite{F} coincide with a special manifold in our sense when the structure is definably complete locally o-minimal.
We first recall Miller's definition of special manifolds.
\begin{definition}[\cite{M2}]\label{def:miller}
We only consider an expansion of the ordered set of reals $(\mathbb R,<)$.
Let $\pi:\mathbb R^n \rightarrow \mathbb R^d$ be a coordinate projection. 
A $d$-dimensional submanifold $X$ of $\mathbb R^n$ is \text{$\pi$-special} if, for each $x \in \pi(X)$, there exists an open box $U$ in $\mathbb R^d$ containing the point $x$ such that each connected component $C$ of $X \cap \pi^{-1}(U)$ projects homeomorphically onto $U$. 
\end{definition}
Note that there are no connected definable sets other than singletons in some ordered structure whose universe is not $\mathbb R$.
For instance, let $\mathbb R_{\text{alg}}$ denote the set of algebraic real numbers.
The structure $(\mathbb R_{\text{alg}},0,1,+,\cdot)$ is o-minimal because sets definable in this structures are semialgebraic by Tarski-Seidenberg principle \cite[Theorem 2.2.1]{BCR} and any nonempty open interval is not connected by \cite[Example 2.4.1]{BCR}.
We can easily derive that a connected definable set is a singleton in this case.
This example illustrates that Miller's definition of special submanifold cannot be extended literally to the non-real cases.

We show that Miller's definition coincides with ours when the structure is a locally o-minimal expansion of the ordered set of reals $(\mathbb R,<)$.
\begin{proposition}\label{prop:miller}
Consider an expansion of the ordered set of reals $(\mathbb R,<)$.
Let $\pi:\mathbb R^n \rightarrow \mathbb R^d$ be a coordinate projection. 
A definable subset of $\mathbb R^n$ is a $\pi$-special submanifold in the sense of Definition \ref{def:miller} if it is a $\pi$-special submanifold in the sense of Definition \ref{def:ours}.
The opposite implication holds true when the structure is locally o-minimal.
\end{proposition}
\begin{proof}
We may assume that $\pi$ is the projection onto the first $d$ coordinates without loss of generality.
We fix a definable subset $X$ of $\mathbb R^n$.

Assume first that $X$ is $\pi$-special submanifold in the sense of Definition \ref{def:ours}.
Fix an arbitrary point $x \in \pi(X)$.
Take an open box $U$ in $\mathbb R^d$ containing the point $x$ and a family $\{V_y\}_{y \in \myfib(X,\pi,x)}$ of open boxes in $\mathbb R^n$ satisfying the conditions in Definition \ref{def:ours}.
Since $\{V_y\}_{y \in \myfib(X,\pi,x)}$ is a family of mutually disjoint open boxes, $X \cap V_y$ are connected components of $X \cap \pi^{-1}(U)$.
They are graphs of continuous maps, and they project homeomorphically onto $U$. 
It means that $X$ is a $\pi$-special submanifold in the sense of Definition \ref{def:miller}.

We next show the opposite implication.
Assume that $X$ is a $\pi$-special submanifold in the sense of Definition \ref{def:miller}.
It is obvious that all the points $x \in X$ are $(X,\pi)$-normal.
Therefore, $X$ is a $\pi$-quasi-special submanifold by Lemma \ref{lem:normal}. 
Fix an arbitrary point $x \in \pi(X)$.
Since $X$ is a quasi-special submanifold, we can take an open box $U$ containing the point $x$ and, for each $y \in X$ with $\pi(y)=x$, there exists an open box $V_y$ such that $y \in V_y$, $\pi(V_y) =U$ and $V_y \cap X$ is the graph of a continuous function defined on $U$.
Shrinking $U$ if necessary, we may assume that any connected component of $X \cap \pi^{-1}(U)$ projects homeomorphically onto $U$ by the assumption.
Consequently, each connected component of $X \cap \pi^{-1}(U)$  is the graph of a continuous map defined on $U$ and contained in some $V_y$.
%It implies that all connected components of $X$ contain points whose image under $\pi$ is $x$.
It means that $X \cap \pi^{-1}(U) \subseteq \bigcup_{y \in \myfib(X,\pi,x)}V_y$.
Therefore, $X$ is a $\pi$-special submanifold in the sense of Definition \ref{def:ours}.
\end{proof}

We next recall the definition of Fornasiero's multi-cell.
\begin{definition}\label{def:Fornasiero}
Let $\mathcal F=(F,<,+,0,\cdot,1,\ldots)$ be an expansion of an ordered commutative field.
Let $X$ be a definable subset of $F^n$ of dimension $d$ and $\pi:F^n \rightarrow F^d$ be a coordinate projection.
Take the permutation $\tau$ of $\{1,\ldots, n\}$ satisfying the conditions (a) and (b) in Definition \ref{def:ours}.
The notation $\overline{\tau}$ denotes the map defined in Definition \ref{def:ours}.
We first consider the case in which $\overline{\tau}^{-1}(X) \subseteq F^d \times (0,1)^{n-d}$. 
A point $a \in F^d$ is \textit{$(X,\pi)$-bad} if it is the projection of a non-$(X,\pi)$-normal point; otherwise, the point $a$ is called \textit{$(X,\pi)$-good}.

Consider the case in which $X$ does not satisfy the previous condition.
Let $\phi:F \rightarrow (0,1)$ be a definable homeomorphism.
Consider the map $$\psi:\operatorname{id}^d \times \phi^{n-d}:F^d \times F^{n-d} \rightarrow F^d \times (0,1)^{n-d}.$$
We say that $a$ is \textit{$(X,\pi)$-good} if it is $(\psi(\overline{\tau}^{-1}(X)),\pi\circ \overline{\tau})$-good.
We define $(X,\pi)$-bad points etc. similarly.

The definable set $X$ is a \textit{$\pi$-multi-cell} if every point of $\pi(X)$ is $(X,\pi)$-good.
\end{definition}
Fornasiero concentrates on the case in which the structure is an expansion of an ordered commutative field.
The following fact is well-known.
We give a proof for the readers' convenience.
\begin{lemma}\label{lem:dctc}
A discrete set definable in a definably complete locally o-minimal expansion of an ordered field is bounded. 
\end{lemma}
\begin{proof}
Let $\mathcal F=(F,<,+,0,\cdot,1,\ldots)$ be a definably complete locally o-minimal expansion of an ordered field.
Let $X$ be a definable discrete subset of $F^n$.
We first demonstrate $X$ is bounded when $n=1$.
Consider the set $Y_+=\{1/x\;|\; x > 0 \text{ and } x \in X\}$.
By local o-minimality, there exists $r>0$ such that the intersection $Z_+=(-r,r) \cap Y_+$ is a finite union of points and open intervals.
If it contains an open interval, $X$ also contains an open interval.
It is a contradiction.
If $Z_+$ is empty, $X \subseteq (-\infty,1/r)$.
Otherwise, there exists the smallest element $s \in Z_+$.
We have $X \subseteq (-\infty,1/s)$ in this case.
We have demonstrated that $X$ is bounded above.
We can prove that $X$ is bounded below in the same manner.
We omit the details.
We have demonstrated that $X$ is bounded when $n=1$.

We consider the case in which $n>1$.
Let $\pi_i:F^n \rightarrow F$ denote the projection onto the $i$-th coordinate for $1 \leq i \leq n$. 
The projection image $\pi_i(X)$ is discrete by Proposition \ref{prop:dim}(1),(6), and it is bounded by this lemma for $n=1$.  
There exists a bounded open interval $I_i$ with $\pi_i(X) \subseteq I_i$ for each $i$.
It is obvious that $X \subseteq \prod_{i=1}^n I_i$.
It implies that $X$ is bounded.
\end{proof}

Fornasiero's multi-cell is equivalent to our special submanifold when the structure is a definably complete locally o-minimal expansion of an ordered field.
\begin{proposition}
Let $\mathcal F=(F,<,+,0,\cdot,1,\ldots)$ be a definably complete locally o-minimal expansion of an ordered field.
Let $\pi:F^n \rightarrow F^d$ be a coordinate projection.
A definable set is a $\pi$-special submanifold in the sense of Definition \ref{def:ours} if and only if it is a $\pi$-multi-cell.
\end{proposition}
\begin{proof}
Let $X$ be a definable subset of $F^n$.
We may assume that $\pi$ is the projection onto the first $d$ coordinates as usual. 
We first demonstrate that, if $X$ is a $\pi$-special submanifold, it is a $\pi$-multi-cell.
Let $\psi:F^d \times F^{n-d} \rightarrow F^d \times (0,1)^{n-d}$ be the definable homeomorphism given in Definition \ref{def:Fornasiero}.
By the definition of $\pi$-special manifold, it is obvious that any point in $\pi(X) \times F^{n-d}$ is $(\psi(X),\pi)$-normal out of the set $\pi(X) \times \partial((0,1)^{n-d})$, where $\partial((0,1)^{n-d})$ denotes the frontier of $(0,1)^{n-d}$.
Fix an arbitrary point $x \in \pi(X)$ and arbitrary $t \in \partial((0,1)^{n-d})$.
We can choose a bounded closed box $B$ in $F^{n-d}$ such that the set $\{y \in F^{n-d}\;|\; (x,y) \in X\}$ is contained in $B$ by Lemma \ref{lem:dctc}.
Thanks to Lemma \ref{lem:equal_size}, expanding $B$ a little bit, we can take an open box $U$ containing the point $x$ such that $X \cap \pi^{-1}(U)$ is contained in $U \times B$.
The image $\psi(B)$ is contained in $(0,1)^{n-d}$ and closed in $F^{n-d}$.
We can take an open box $V$ in $F^{n-d}$ containing the point $t$ and having an empty intersection with $\psi(B)$.
The intersection $(U \times V) \cap \psi(X)$ is empty.
It means that $(x,t)$ is $(\psi(X),\pi)$-normal.
We have demonstrated that $X$ is $\pi$-multi-cell.

We next show the opposite implication.
Assume that $X$ is a $\pi$-multi-cell.
As an initial step, we show that $(X,\pi)$ is locally bounded at every point in $\pi(X)$.
Fix a point $x \in \pi(X)$.
Assume for contradiction that $(\mathcal B_d(x,t) \times F^{n-d}) \cap X$ is not bounded for any $t>0$.
It implies that the set $$S(t):=(\mathcal B_d(x,t) \times (F^{n-d} \setminus \mathcal B_{n-d}(O,1/t))) \cap X$$ is not empty for any $t>0$, where $O$ denotes the origin of $F^d$.
By Lemma \ref{lem:definable_choice}, we can find a definable function $f:(0,\infty) \rightarrow F^n$ such that $f(t)$ is an element of $S(t)$ for $t>0$.
%We may assume that the restriction of $f$ to an open interval $(0,\delta)$ is continuous by Proposition \ref{prop:limit}(1).
There exists a unique limit $z=\lim_{t \to 0}\psi(f(t)) \in \partial ((0,1)^{n-d})$ by Proposition \ref{prop:limit}(2).
The point $(x,z) \in F^d \times F^{n-d}$ is in the frontier of $\psi(X)$.
In particular, it is not contained in $\psi(X)$, and any open box containing it has a nonempty intersection with $\psi(X)$.
It implies that the point $(x,z)$  is not $(\psi(X),\pi)$-normal.
It contradicts the assumption that $X$ is a $\pi$-multi-cell.
We have demonstrated that $(X,\pi)$ is locally bounded at every point in $\pi(X)$.

It is obvious that any point in $X$ is $(X,\pi)$-normal.
The $\pi$-multi-cell $X$ is a $\pi$-quasi-submanifold by Lemma \ref{lem:normal}.
It is obvious that $X$ is closed in $\pi^{-1}(\pi(X))$ from the definition of multi-cells.
The set $X$ is a $\pi$-special submanifold by Lemma \ref{lem:when}.
\end{proof}

\subsection{Decomposition into special submanifolds}
We finally demonstrate the main theorem which asserts that any definable set is decomposed into finitely many special submanifolds.
We consider the following cases separately.
\begin{enumerate}
\item[(A)] There exists a definable, discrete, closed and unbounded subset of $M$;
\item[(B)] The structure is a model of DCTC.
\end{enumerate}
We first consider the case (A).

\begin{lemma}\label{lem:succ}
Let $\mathcal M=(M,<,\ldots)$ be a definably complete structure.
Let $X$ be a definable discrete closed subset of $M$ such that $\sup(X)=\infty$.
Then, there exists a definable map $\mysucc:X \rightarrow X$ such that $x<\mysucc(x)$ and there are no elements in $X$ between $x$ and $\mysucc(x)$.
\end{lemma}
\begin{proof}
Define $\mysucc(x)=\inf\{y \in X\;|\;y>x\}$.
We have $\mysucc(x) \in X$ because $X$ is closed.
We get $x<\mysucc(x)$ because $X$ is discrete.
It is obvious that there are no elements in $X$ between $x$ and $\mysucc(x)$.
\end{proof}

\begin{lemma}\label{lem:case1}
Consider a definably complete locally o-minimal expansion of an ordered group $\mathcal M=(M,<,0,+,\ldots)$.
Assume that there exists a definable, discrete, closed and unbounded subset $D$ of $M$.  
Let $\pi:M^n \rightarrow M^d$ be a coordinate projection.
A $\pi$-quasi-special submanifold $X$ of $M^n$ which is closed in $\pi^{-1}(\pi(X))$ is a $\pi$-special submanifold.
\end{lemma}
\begin{proof}
We may assume that $\pi$ is the projection onto the first $d$ coordinate without loss of generality.
We may further assume that $\sup(D)=\infty$ and $\inf(D)=-\infty$ by replacing $D$ with $D \cup \{-x \;|\; x \in D\}$.
Set $e=n-d$.
Let $\mysucc:D \rightarrow D$ denote the successor map given in Lemma \ref{lem:succ}.
For any $x=(x_1,\ldots, x_e) \in D^e$, we set 
$$\mathcal C(x)=\{(y_1,\ldots, y_e) \in M^e\;|\; x_i \leq y_i \leq \mysucc(x_i) \text{ for all } 1 \leq i \leq e\}.$$
The set $\mathcal C(x)$ is a bounded closed box.

We fix an arbitrary point $a$ in $\pi(X)$.
Take a bounded open box $B$ such that $B$ contains the point $a$ and the closure $\mycl(B)$ is contained in $\pi(X)$.
It is possible because $\pi(X)$ is open.
Set $E=\myfib(x,\pi,a)=\{y \in M^e\;|\; (a,y) \in X\}$.
Fix a sufficiently small positive element $\varepsilon \in M$.
By Lemma \ref{lem:normal}, we may assume that $X \cap (B \times \mathcal B_e(y,\varepsilon))$ is the graph of a continuous function defined on $B$ for each $y \in E$ by shrinking $B$ if necessary.
Set $$Y=((\mycl(B) \times M^e) \cap X) \setminus \bigcup_{y \in E} B \times \mathcal B_e(y,\varepsilon).$$
It is closed because $X$ is closed in $\pi(X)\times M^e=\pi^{-1}(\pi(X))$ and $\mycl(B) \subseteq \pi(X)$.
If $Y$ is an empty set, the pair $(B,\{B \times \mathcal B_e(y,\varepsilon)\}_{y \in E})$ satisfies the conditions (1) through (3) in Definition \ref{def:ours}.

We next consider the case in which $Y \neq \emptyset$.
Set $F=\{x \in D^e\;|\; (\mycl(B) \times \mathcal C(x)) \cap Y \neq \emptyset\}$.
It is not empty.
In addition, we have $\bigcup_{x \in F} (\mycl(B) \times \mathcal C(x)) \cap Y=Y$ by the definition of $D$.
Define the function $\rho:F \rightarrow M$ so that $\rho(x)$ is the distance of the singleton $\{a\}$ to the definable set $\pi( (\mycl(B) \times \mathcal C(x)) \cap Y)$.
The set $\pi( (\mycl(B) \times \mathcal C(x)) \cap Y)$ does not contain the point $a$ by the definition of $Y$.
It is closed and bounded by \cite[Lemma 1.7]{M} because $(\mycl(B) \times \mathcal C(x)) \cap Y$ is bounded and closed.
They imply that $\rho(x)>0$ for all $x \in F$ by Lemma \ref{lem:dist}.
The definable set $F$ is of dimension zero by Proposition \ref{prop:dim}(1),(2),(4).
The image $\rho(F)$ is of dimension zero and it is discrete and closed by Proposition \ref{prop:dim}(1),(6).
In particular, we have $\inf \rho(F) \in \rho(F)$ and $\inf \rho(F) >0$.
Take $\delta>0$ so that $\delta <\inf \rho(F)$ and the box $\mathcal B_d(a,\delta)$ is contained in $B$.
It is obvious that $\pi^{-1}(\mathcal B_d(a,\delta)) \cap X \subseteq \bigcup_{y \in E} \mathcal B_d(a,\delta) \times \mathcal B_e(y,\varepsilon)$.
It implies that the pair $(\mathcal B_d(a,\delta),\{\mathcal B_d(a,\delta)) \times \mathcal B_e(y,\varepsilon)\}_{y \in E})$ satisfies the conditions (1) through (3) in Definition \ref{def:ours}.
\end{proof}

We next treat the case (B).
\begin{lemma}\label{lem:case2}
Consider a model of DCTC $\mathcal M=(M,<,0,+,\ldots)$.
Let $\pi:M^n \rightarrow M^d$ be a coordinate projection and $X$ be a $\pi$-quasi-special submanifold $X$ of $M^n$ which is closed in $\pi^{-1}(\pi(X))$.
The definable set $\mynlb(X,\pi)$ given by $$\mynlb(X,\pi)=\{x \in \pi(X)\;|\; (X,\pi) \text{ is not locally bounded at }x\}$$ has an empty interior.
\end{lemma}
\begin{proof}
It is obvious that $\mynlb(X,\pi)$ is definable.
We omit the details.
The remaining task is to show that it has an empty interior.

Assume for contradiction that $\mynlb(X,\pi)$ contains a nonempty open box $B$.
As usual, we may assume that $\pi$ is the projection onto the first $d$ coordinates.
Set $e=n-d$.
Let $\rho_i:M^e \rightarrow M$ be the projection to the $i$-th coordinate for $1 \leq i \leq e$.
Set $Y_{i,x}=\rho_i(\myfib(X,\pi,x))=\rho_i(\{y \in M^e\;|\; (x,y) \in X\})$ for $x \in \pi(X)$ and $1 \leq i \leq e$.
By the definition of quasi-special submanifolds, the set $\myfib(X,\pi,x)$ is discrete.
The sets $Y_{i,x}$ are discrete and closed by Proposition \ref{prop:dim}(1),(6).
They are bounded by the definition of a model of DCTC.
Consider the functions $u_i,l_i:B \rightarrow M$ given by $l_i(x)=\inf Y_{i,x}$ and $u_i(x)=\sup Y_{i,x}$ for $1 \leq i \leq e$.
Shrinking $B$ if necessary, we may assume that $u_i$ and $l_i$ are continuous by Proposition \ref{prop:dim}(7).

Fix a point $a \in B$.
Take a closed box $C$ such that $a \in \myint(C)$ and $C \subseteq B$.
There are $L_i$ and $U_i$ in $M$ such that $L_i<l_i(x) \leq u_i(x) <U_i$ for all $x \in C$ and $1 \leq i \leq e$ by \cite[Proposition 1.10]{M}.
By the definitions of $L_i$ and $U_i$, we have $\pi^{-1}(\myint(C)) \cap X \subseteq \myint(C) \times (\prod_{i=1}^e (L_i,U_i))$.
It means that $(X,\pi)$ is locally bounded at the point $a$.
It is a contradiction because $a \in \mynlb(X,\pi)$.
We have demonstrated that $\mynlb(X,\pi)$ has an empty interior.
\end{proof}

The following is the main part of the proof:
\begin{lemma}\label{lem:decomposition}
Consider a definably complete locally o-minimal expansion of an ordered group $\mathcal M=(M,<,0,+,\ldots)$.
Let $X$ be a definable subset of $M^n$.
There exists a family $\{C_i\}_{i=1}^N$ of mutually disjoint special submanifolds with $X= \bigcup_{i=1}^N C_i$.
Furthermore, the number $N$ of special submanifolds is not greater than the number uniquely determined only by $n$.
\end{lemma}
\begin{proof}
By \cite[Lemma 4.3]{Fuji4}, the definable set $X$ is decomposed into finitely many mutually disjoint quasi-special submanifolds and the number of the quasi-special submanifolds is not greater than the number uniquely determined only by $n$.
Therefore, we may assume that $X$ is a $\pi$-quasi-special submanifold of dimension $d$, where $\pi:M^n \rightarrow M^d$ is a coordinate projection.
 
We prove the lemma by induction on $d$.
The definable set $X$ is obviously a special submanifold when $d=0$.
We consider the case in which $d>0$.
The frontier $\partial X$ of $X$ is of dimension $< d$ by Proposition \ref{prop:dim}(8).
When $X$ is closed, set $Z_1=\emptyset$ and $X_1=X$.
It is obvious that $X_1$ is closed in $\pi^{-1}(\pi(X_1))$ in this case.
We next consider the case $X$ is not closed.
The closure of the image $Y_1:=\mycl(\pi(\partial X))$ is of dimension $<d$ by Proposition \ref{prop:dim}(5),(6),(8).
The intersection $\pi^{-1}(x) \cap X$ is discrete by the definition of quasi-special submanifolds for each $x \in Y_1$.
It is of dimension zero by Proposition \ref{prop:dim}(1).
Set $Z_1=\pi^{-1}(Y_1) \cap X$ and $X_1=X \setminus Z_1$.
We have $\dim Z_1= \dim Y_1 + \dim (\pi^{-1}(x) \cap X) = \dim Y_1 <d$ for any $x \in Y_1$ by Proposition \ref{prop:dim}(9).
Since each point in $X_1$ is $(X_1,\pi)$-normal, $X_1$ is a quasi-special submanifold by Lemma \ref{lem:normal}.
It is obvious that $X_1$ is closed in $\pi^{-1}(\pi(X_1))$.

We treat two separate cases.
We first consider the case in which there exists an unbounded definable discrete subset of $M$.
The quasi-special submanifold $X_1$ is a special submanifold by Lemma \ref{lem:case1}.
By the induction hypothesis, $Z_1$ is decomposed into mutually disjoint special submanifolds $C_1, \ldots, C_N$.
The decomposition $X=X_1 \cup \bigcup_{i=1}^N C_i$ is the desired decomposition.

The latter case is the case in which all definable discrete subsets of $M$ are bounded.
The structure $\mathcal M$ is a model of DCTC by Definition \ref{def:dctc}.
The definable set $\mynlb(X_1,\pi)$ is of dimension $<d$ by Lemma \ref{lem:case2}.
Set $Y_2=\mycl(\mynlb(X_1,\pi))$.
We have $\dim Y_2<d$ by Proposition \ref{prop:dim}(5),(8).
Set $Z_2=\pi^{-1}(Y_2) \cap X_1$ and $X_2=X_1 \setminus Z_2$.
The pair $(X_2,\pi)$ is obviously locally bounded at every point in $\pi(X_2)$.
Apply the same argument for $X_1$ and $Z_1$ to $X_2$ and $Z_2$.
The definable set $X_2$ is a quasi-special submanifold which is closed in $\pi^{-1}(\pi(X_2))$ and $\dim Z_2<d$.
We obtain $\dim Z_1 \cup Z_2 <d$ by Proposition \ref{prop:dim}(5).
The definable set $X_2$ is a special submanifold by Lemma \ref{lem:when}.
By the induction hypothesis, $Z_1 \cup Z_2$ is decomposed into mutually disjoint special submanifolds $C_1, \ldots, C_N$.
The decomposition $X=X_2 \cup \bigcup_{i=1}^N C_i$ is the desired decomposition.
\end{proof}

\begin{definition}\label{def:decomposition}
Consider an expansion of a densely linearly order without endpoints $\mathcal M=(M,<,\ldots)$.
Let $\{X_i\}_{i=1}^m$ be a finite family of definable subsets of $M^n$.
A \textit{decomposition of $M^n$ into special submanifolds partitioning $\{X_i\}_{i=1}^m$} is a finite family of special submanifolds $\{C_i\}_{i=1}^N$ such that 
\begin{itemize}
\item $\bigcup_{i=1}^NC_i =M^n$, 
\item $C_i \cap C_j=\emptyset$ when $i \not=j$ and 
\item either $C_i$ has an empty intersection with $X_j$ or it is contained in $X_j$
\end{itemize}
for any $1 \leq i \leq m$ and $1 \leq j \leq N$.
A decomposition $\{C_i\}_{i=1}^N$ of $M^n$ into special submanifolds \textit{satisfies the frontier condition} if the closure of any special manifold $\mycl(C_i)$ is the union of a subfamily of the decomposition.
\end{definition}

\begin{theorem}\label{thm:frontier_condition}
Consider a definably complete locally o-minimal expansion of an ordered group $\mathcal M=(M,<,0,+,\ldots)$.
Let $\{X_i\}_{i=1}^m$ be a finite family of definable subsets of $M^n$.
There exists a decomposition $\{C_i\}_{i=1}^N$ of $M^n$ into special submanifolds partitioning $\{X_i\}_{i=1}^m$ and satisfying the frontier condition.
Furthermore, the number $N$ of special submanifolds is not greater than the number uniquely determined only by $m$ and $n$.
\end{theorem}
\begin{proof}
The proof is literally same as the proof of \cite[Theorem 4.4, Theorem 4.5]{Fuji4} except that we use Lemma \ref{lem:decomposition} instead of \cite[Lemma 4.3]{Fuji4}.
We omit the details.
\end{proof}

\begin{remark}
Lemma \ref{lem:equal_size} implies that there exists a family $\{U_y\}_{y \in \myfib(X,\pi,x)}$ of open boxes $U_y$ parameterized by the definable set $\myfib(X,\pi,x)$ such that 
\begin{enumerate}
\item[(a)]  the union $\bigcup_{y \in \myfib(X,\pi,x)} \{y \} \times U_y$ is definable and \item[(b)] $X \cap U_y$ is the graph of a definable map for each $y \in \myfib(X,\pi,x)$.
\end{enumerate}
This fact is essentially used in our proof of the decomposition into special submanifolds.

Consider a definably complete locally o-minimal structure $\mathcal M=(M,<,\ldots)$ which is not necessarily an expansion of an ordered group.
A definable, discrete closed subset $D$ in $M$ is always a special submanifold.
In this case, a family satisfying the conditions (a) and (b) is a family $\{I_x\}_{x \in D}$ of mutually disjoint open intervals $I_x$ containing the point $x$ such that $\bigcup_{x \in D} \{x\} \times I_x$ is definable.
The author does not know whether such a family $\{I_x\}_{x \in D}$ exists when a definable choice lemma such as Lemma \ref{lem:definable_choice} is unavailable.
He does not know whether Theorem \ref{thm:frontier_condition} still holds true when we drop the assumption that the structure is an expansion of an ordered group, neither.
\end{remark}

\subsection{Decomposition into special submanifolds with tubular neighborhoods}
A tubular neighborhood of a submanifold in a Euclidean space is a convenient tool for geometric studies of semialgebraic sets \cite{BCR} and others \cite{Shiota}.
We define a special submanifolds with a tubular neighborhood and give a decomposition theorem into special submanifolds with tubular neighborhoods for future use.

We define a special submanifold with a tubular neighborhood.
\begin{definition}\label{def:tub1}
Let $\mathcal M=(M,<,+,0,\ldots)$ be an expansion of an ordered abelian group.
Let $X$ be a $\pi$-special submanifold in $M^n$, where $\pi:M^n \rightarrow M^d$ is a coordinate projection.
Let $\tau$ be the unique permutation of $\{1,\ldots, n\}$ satisfying the conditions (a) and (b) in Definition \ref{def:ours}.
Set $U=\pi(X)$.

When $\dim X<n$, the tuple $(X,\pi,T,\eta,\rho)$ is a \textit{special submanifold with a tubular neighborhood} if 
\begin{enumerate}
\item[(a)] $T$ is a definable open subset of $\pi^{-1}(U)$;
\item[(b)] $\eta:U \rightarrow F$ is a positive bounded definable continuous function such that, for all $u \in U$, we have 
$$\myfib(T,\pi,u)= \bigcup_{x \in \myfib(X,\pi,u)}  \mathcal B_{n-d}(x,\eta(u))$$
 and $$\mathcal B_{n-d}(x_1,\eta(u)) \cap \mathcal B_{n-d}(x_2,\eta(u)) = \emptyset$$ for all $x_1,x_2 \in \myfib(X,\pi,u)$ with $x_1 \neq x_2$;
\item[(c)] $\rho:T \rightarrow X$ is a definable continuous retraction such that, for any $u \in U$, we have $\rho(\pi^{-1}(u) \cap T) \subseteq \pi^{-1}(u) \cap X$ and $\rho(\overline{\tau}(u,y)))=\overline{\tau}(u,x)$ for all $x \in \myfib(X,\pi,u)$ and $y \in \mathcal B_{n-d}(x,\eta(u))$.
\end{enumerate}
When $\dim X=n$, the tuple $(X,\pi,T,\eta,\rho)$ is a \textit{special submanifold
 with a tubular neighborhood} if $X$ is open, $T=X$, $\eta \equiv 0$, and $\rho$ is the identity map on $X$.
 
A \textit{decomposition of $M^n$ into special submanifolds with tubular neighborhoods} is a finite family of special submanifolds with tubular neighborhoods $\{(X_i,\pi_i,T_i,\eta_i,\rho_i)\}_{i=1}^N$ such that $\{(X_i,\pi_i)\}_{i=1}^N$ is a decomposition of $M^n$ into special submanifolds.
We say that a decomposition $\{(X_i,\pi_i,T_i,\eta_i,\rho_i)\}_{i=1}^N$ of $M^n$ into  special submanifolds with tubular neighborhoods partitions a given finite family of definable sets and satisfies the frontier condition if so does the decomposition into special submanifolds $\{(X_i,\pi_i)\}_{i=1}^N$.
\end{definition}

This definition seems to be technical, but a decomposition into special submanifolds with tubular neighborhoods is useful.
The following theorem guarantees the existence of the decomposition.

\begin{theorem}\label{thm:tubular_decom}
Let $\mathcal M=(M,<,+,0,\ldots)$ be a definably complete locally o-minimal expansion of an ordered group.
Let $\{X_i\}_{i=1}^m$ be a finite family of definable subsets of $M^n$.
There exists a decomposition of $M^n$ into special submanifolds with tubular neighborhoods partitioning $\{X_i\}_{i=1}^m$ and satisfying the frontier condition.
In addition, the number of special submanifolds with tubular neighborhoods is bounded by a function of $m$ and $n$.
\end{theorem}
\begin{proof}
We first demonstrate the following claim:
\medskip

\textbf{Claim 1.} Let $\pi:M^n \rightarrow M^d$ be a coordinate projection and $C \subseteq M^n$ be a $\pi$-special submanifold.
There exists a special submanifolds with tubular neighborhoods $(X,\pi,T,\eta,\rho)$ such that $X \subseteq C$ and $\dim C \setminus X<d$.
\medskip

The claim is obvious when $d=n$.
We have only to set $X=T=C$, $\eta=0$ and $\rho=\operatorname{id}$.
We next consider the case in which $d<n$.
For simplicity of notations, we assume that $\pi$ is the coordinate projection onto the first $d$ coordinates.
Note that $\myfib(C,\pi,u)$ are discrete and closed for all $u \in U=\pi(C)$ by the definition of special submanifolds and Proposition \ref{prop:dim}(1).
We also note that $\myfib(C,\pi,u) \setminus \{x\}$ is also closed and discrete for any $x \in \myfib(C,\pi,u)$.
Take a positive element $c \in M$.
Consider the map $\eta':U \rightarrow F$ defined by
$$\eta'(u)=\dfrac{1}{3}\inf\left( \{c\} \cup \bigcup_{ x \in \myfib(C,\pi,u)} \{|y-x| \;|\; y \in \myfib(C,\pi,u) \setminus \{x\}\}\right).$$
The map $\eta'$ is obviously definable.
For a fixed $u \in U$, the set in the round brackets is a discrete and closed set contained in the open interval $(0,\infty)$ by Proposition \ref{prop:dim}(1),(6).
In particular, we have $\eta'(u)>0$ for all $u \in U$.

Let $D$ be the closure of the set of points at which $\eta'$ is discontinuous.
We have $\dim D<d$ by Proposition \ref{prop:dim}(5),(7),(8).
Set $V=U \setminus D$, which is a definable open subset of $M^d$.
We define $X$, $T$, $\eta$ and $\rho$ as follows:
\begin{itemize}
\item $X=C \cap \pi^{-1}(V)$;
\item $T=\bigcup_{u \in V}\bigcup_{x \in \myfib(C,\pi,u)} \mathcal \{u\} \times \mathcal B_{n-d}(x,\eta'(u))$;
\item $\eta$ is the restriction of $\eta'$ to $V$;
\item $\rho:T \rightarrow X$ is the map such that $\rho(x)$ is the unique $y \in X$ with $u=\pi(x)=\pi(y)$ and $\Pi(x) \in B_{n-d}(\Pi(y),\eta'(u))$, where $\Pi$ is the projection of $M^n$ forgetting the first $d$-coordinates.
\end{itemize}
We show that $(X,\pi,T,\eta,\rho)$ is a special submanifold with a tubular neighborhood.
They obviously satisfy the conditions (a) and (b) in Definition \ref{def:tub1}.
It is also obvious that $\rho$ is definable and satisfies the inclusion and the equality given in Definition \ref{def:tub1}(c).
The remaining task is to demonstrate that $\rho$ is continuous.

Take a point $x \in T$.
Set $u=\pi(x)$.
There exists a unique point $y \in X$ such that $\pi(y)=u$ and $\Pi(x) \in \mathcal B_{n-d}(\Pi(y),\eta'(u))$.
There exist an open box $W$ containing the point $y$ and a definable continuous map $\zeta$ defined on $\pi(W)$ such that $W \cap C$ is the graph of $\zeta$.
We may assume that $\pi(W)$ is contained in $V$, shrinking $W$ if necessary.
Set $W'=\bigcup_{u \in \pi(W)}\mathcal \{u\} \times \mathcal B_{n-d}(\zeta(u),\eta'(u))$.
The set $W'$ is an open subset of $T$ containing the point $x$.
For $(u',t) \in W'$, we get $\rho(u',t)=(u',\zeta(u'))$.
This equality implies that $\rho$ is continuous on $W'$.
Therefore, the map $\rho$ is continuous on its domain of definition.

The set $C \setminus X$ is given by $C \cap \pi^{-1}(D)$.
Since $\dim D<d$ and $\dim \pi^{-1}(u) \cap C =0$ for all $u \in D$, we get $\dim C \setminus X <d$ by Proposition \ref{prop:dim}(9).
We have now demonstrated the claim.
\medskip

For any definable subset $X$ of $M^n$, we define a decomposition of $X$ into special submanifolds in the same manner as the case in which $X=M^n$.
We next demonstrate the following claim.
The theorem is a direct corollary of this claim.
\medskip

\textbf{Claim 2.}
Let $X$ be a definable subset of $M^n$ and $\{X_i\}_{i=1}^m$ be a finite family of definable subsets of $X$.
There exists a decomposition of $X$ into special submanifolds with tubular neighborhoods partitioning $\{X_i\}_{i=1}^m$ and satisfying the frontier condition.
In addition, the number of special submanifolds with tubular neighborhoods is bounded by a function of $m$ and $n$.
\medskip

We prove the claim by induction on $d=\dim X$.
We first consider the case in which $d=0$.
Apply Theorem \ref{thm:frontier_condition} and find a decomposition of $M^n$ into special submanifolds partitioning $\{X\} \cup \{X_i\}_{i=1}^m$.
A subfamily of this partition gives a decomposition of $X$ into special submanifolds $\{C_i\}_{i=1}^N$ partitioning $\{X_i\}_{i=1}^m$.
Since $\dim C_i=0$, the special submanifolds $C_i$ are closed for all $1 \leq i \leq N$ by Proposition \ref{prop:dim}(1).
Apply Claim 1 to the special submanifold $C_i$ for each $1 \leq i \leq N$, and we get a special submanifold with tubular neighborhoods $(C_i,\pi_i,T_i,\eta_i,\rho_i)$.
The family $\{(C_i,\pi_i,T_i,\eta_i,\rho_i)\}_{i=1}^N$ is a desired decomposition.

We next consider the case in which $d>0$.
Apply Theorem \ref{thm:frontier_condition}.
We get a decomposition of $X$ into special submanifolds $\{C_i\}_{i=1}^N$ partitioning $\{X_i\}_{i=1}^m$.
We may assume that $\dim C_i=d$ for all $1 \leq i \leq L$ and $\dim C_i<d$ for all $i>L$ without loss of generality.
Apply Claim 1 to the special submanifold $C_i$ for each $1 \leq i \leq L$, we get a definable subset $C'_i$ of $C_i$ with $\dim C_i \setminus C'_i<d$ and a special submanifold with a tubular neighborhood $(C'_i,\pi_i,T_i,\eta_i,\rho_i)$.
Set $D_i=C_i \setminus C'_i$ for $1 \leq i \leq L$ and $D_i = C_i$ for $L<i \leq N$.
We put $X'=\bigcup_{i=1}^N D_i$.
We obtain $\dim X'<d$ by Proposition \ref{prop:dim}(5).
Apply the induction hypothesis to $X'$, then there exists a decomposition $\{(E_i,\pi'_i,T'_i,\eta'_i,\rho'_i)\}_{i=1}^{N'}$ of $X'$ into special submanifolds with tubular neighborhoods partitioning $\{D_i\}_{i=1}^N \cup \{D_i \cap \partial C'_j\}_{1 \leq i \leq N,\ 1 \leq j \leq L}$ and satisfying the frontier condition.
The family $\{(C'_i,\pi_i,T_i,\eta_i,\rho_i)\}_{i=1}^L \cup \{(E_i,\pi'_i,T'_i,\eta'_i,\rho'_i)\}_{i=1}^{N'}$ is a desired decomposition.

The `in addition' part is clear from the proof.
\end{proof}

\begin{remark}
Let $r$ be a positive integer.
When the structure is a definably complete locally o-minimal expansion of an ordered field, Theorem \ref{thm:tubular_decom} can be easily extended to a decomposition into special $\mathcal C^r$-submanifolds with tubular neighborhoods by minor modifications of the definition of the neighborhood $\mathcal B_m(x,r)$ and proofs.

Several assertions demonstrated in the o-minimal setting using the definable cell decomposition theorem and the stratification theorem also hold true in definably complete locally o-minimal expansions of ordered fields.
We use decomposition into special $\mathcal C^r$-submanifolds with tubular neighborhoods instead of the definable cell decomposition theorem and the stratification theorem.
They will be proved in the forthcoming paper.
\end{remark}

\end{document}